\documentclass[11pt]{article}

\usepackage{mathrsfs}
\usepackage{amsmath}
\usepackage{amssymb}
\usepackage{amsthm}
\usepackage{txfonts}
\usepackage{amsfonts}
\usepackage{latexsym}
\usepackage{indentfirst}

\usepackage{wrapfig}

\newcommand\RR{\mathbb{R}}
\newcommand\Rd{\mathbb{R}^d}
\newcommand\NN{\mathbb{N}}

\newcommand\PP{\mathbb{P}}
\newcommand\Mean{\text{E}}

\newcommand\vx{\boldsymbol{x}}   
\newcommand\vy{\boldsymbol{y}}   
\newcommand\vz{\boldsymbol{z}}
\newcommand\vc{\boldsymbol{c}}
\newcommand\vxi{\boldsymbol{\xi}}

\newcommand\vk{\boldsymbol{k}}
\newcommand\valpha{\boldsymbol{\alpha}}
\newcommand\vbeta{\boldsymbol{\beta}}

\newcommand\vvarphi{\boldsymbol{\varphi}}
\newcommand\veta{\boldsymbol{\eta}}
\newcommand\ve{\boldsymbol{e}}
\newcommand\vrho{\boldsymbol{\rho}}
\newcommand\vf{\boldsymbol{f}}
\newcommand\vb{\boldsymbol{b}}

\newcommand\vv{\boldsymbol{v}}
\newcommand\vt{\boldsymbol{t}}
\newcommand\vzeta{\boldsymbol{\zeta}}
\newcommand\hatvf{\boldsymbol{\hat{f}}}
\newcommand\hatvc{\boldsymbol{\hat{c}}}

\newcommand\vL{\boldsymbol{\mathcal{L}}}
\newcommand\vS{\boldsymbol{S}}
\newcommand\vY{\boldsymbol{Y}}

\newcommand\vA{\mathsf{A}}

\newcommand\ud{\textup{d}}
\newcommand\Normal{\mathcal{N}}
\newcommand\Order{\mathcal{O}}
\newcommand\Span{\text{span}}

\newcommand\Filter{\mathcal{F}}
\newcommand\Domain{\mathcal{D}}
\newcommand\Aset{\mathcal{A}}

\newcommand\Var{\text{Var}}
\newcommand\Cov{\text{Cov}}

\newcommand\Det{\text{det}}

\newcommand\Ker{\mathcal{K}}
\newcommand\rca{\mathrm{rca}}
\newcommand\range{\mathrm{range}}

\newcommand\Hilbert{\mathcal{H}}
\newcommand\Banach{\mathcal{B}}
\newcommand\Cont{\mathrm{C}}
\newcommand\Leb{\mathrm{L}}

\newcommand{\norm}[1]{\left\lVert#1\right\rVert}
\newcommand{\abs}[1]{\left\lvert#1\right\rvert}

\def\XXint#1#2#3{{\setbox0=\hbox{$#1{#2#3}{\int}$ }
\vcenter{\hbox{$#2#3$ }}\kern-.6\wd0}}

\newtheorem{theorem}{Theorem}[section]

\newtheorem{corollary}[theorem]{Corollary}
\newtheorem{proposition}[theorem]{Proposition}

\theoremstyle{definition}
\newtheorem{definition}[theorem]{Definition}
\newtheorem{example}[theorem]{Example}

\theoremstyle{remark}

\numberwithin{equation}{section}
\numberwithin{figure}{section}
\numberwithin{table}{section}

\begin{document}



\title{Kernel-based Approximation Methods for Generalized Interpolations: \\ A Deterministic or Stochastic Problem?}

\author{Qi Ye \\ 
School of Mathematical Sciences, South China Normal University, \\ Guangzhou, Guangdong, China, 510631 \\
Email: yeqi@m.scnu.edu.cn \\
}


\date{}

\maketitle

\begin{abstract}
In this article, we solve a deterministically generalized interpolation problem by a stochastic approach.
We introduce a kernel-based probability measure on a Banach space by a covariance kernel which is defined on the dual space of the Banach space. The kernel-based probability measure provides a numerical tool to construct and analyze the kernel-based estimators conditioned on non-noise data or noisy data including algorithms and error analysis. Same as meshfree methods, we can also obtain the kernel-based approximate solutions of elliptic partial differential equations by the kernel-based probability measure.


\end{abstract}

%



\section{Introduction}\label{sec:Intr}

Kernel-based approximation methods are fundamental approaches of meshfree methods in \cite{Buhmann2003,Fasshauer2007,Wendland2005} and statistical learning in \cite{HastieTibshiraniFriedman2009,ScholkopfSmola2002,SteinwartChristmann2008}.
The kernel-based approximation methods are
known by a variety of names in the monographs including scattered data approximation, radial basis function, kernel-based collocation, smoothing spline, Gaussian process regression, and kriging.
The papers \cite{ScheuererSchabackSchlather2012,Wu1992} and the book \cite{FasshauerMcCourt2015} presented that the kernel-based estimators of deterministic
and stochastic interpolations had mathematically equivalent formulas.
In the papers \cite{Ye2016PDK,Ye2017Kriging,Ye2017PDE}, we combined the theory and knowledge of numerical analysis, regression analysis, and stochastic analysis to introduce a concept of kernel-based probability measures on Sobolev spaces to improve the kernel-based approximation methods in high-dimensional interpolations and numerical solutions of partial differential equations.

In this article, we generalize the definitions and theorems of original kernel-based probability measures such that we can construct the kernel-based estimators of generalized interpolations in Banach spaces.
Different from the classical numerical methods, we will solve a deterministically generalized interpolation problem by a stochastic approach.
Now we show the main idea by a simple example.
For the standard interpolation problem, we have the data
\[
(\vx_1,f_1),\ldots,(\vx_n,f_n)\in\Domain\times\RR,
\]
where $\Domain$ is a domain of $\Rd$. By the meshfree methods, for another location $\vx\in\Domain$, we use a positive definite kernel $K:\Domain\times\Domain\to\RR$ to construct the kernel basis
\[
K(\vx,\vx_1),\ldots,K(\vx,\vx_n),
\]
to approximate the unknown value at $\vx$ dependent of the interpolation conditions. For the Hermite-Birkhoff interpolation problem, we have the data
\[
\left(\delta_{\vx_1}\circ\Delta,f_1\right),\ldots,\left(\delta_{\vx_n}\circ\Delta,f_n\right)
\in\Cont^{2}(\Domain)^{\ast}\times\RR,
\]
where $\delta_{\vx}$ is the point evaluation function at $\vx$, $\Delta$ is the Laplace operator, and $\Cont^{2}(\Domain)^{\ast}$ is the dual space of $\Cont^{2}(\Domain)$.
If $K\in\Cont^{4}(\Domain\times\Domain)$, then we have the kernel basis
\[
\Delta_{\vy}K(\vx,\vy)|_{\vy=\vx_1},\ldots,
\Delta_{\vy}K(\vx,\vy)|_{\vy=\vx_n},
\]
to construct the estimator. For the generalized interpolation problem discussed here, we have the data
\[
(L_1,f_1),\ldots,(L_n,f_n)\in\Banach^{\ast}\times\RR,
\]
where $\Banach^{\ast}$ is the dual space of a Banach space $\Banach$.
By the probability theory, we introduce a kernel-based probability measure $\PP_{\Ker}^{\mu}$ on the Banach space $\Banach$ with a mean $\mu\in\Banach$ and a covariance kernel $\Ker:\Banach^{\ast}\times\Banach^{\ast}\to\RR$.
Moreover, we will compute the estimator by the average over the interpolation event $\Aset_n:=\left\{\omega\in\Banach:L_1\omega=f_1,\ldots,L_n\omega=f_n\right\}$ measured by $\PP_{\Ker}^{\mu}$.
If $\delta_{\vx}\in\Banach^{\ast}$, then the estimator is a linear combination of the kernel basis
\[
\Ker(\delta_{\vx},L_1),\ldots,\Ker(\delta_{\vx},L_n),
\]
where its coefficients are solved by a related linear system.
In meshfree methods, the inner products of reproducing kernel Hilbert spaces are used to compute the kernel-based estimators from deterministic data.
In kriging methods, the covariances of random variables are used to compute the kernel-based estimators from random data.
By using the kernel-based probability measures on Banach spaces,
we connect the dual-bilinear products of Banach spaces and the covariances of Gaussian processes to construct the kernel-based estimators from the deterministic data.
This indicates that approximation theory could be redone in a stochastic framework.

Finally, we give the outlines of this article.
In Section \ref{sec:KerProbSpace}, we introduce kernel-based probability measures on Banach spaces. Definition \ref{def:KerProbSpace} shows that the kernel-based probability measure is primarily dependent on the covariance kernels. Moreover, we illustrate the construction of covariance kernels by the positive definite kernels as shown in Example~\ref{exa:PDK-Sobolev}. The covariance kernels can be viewed as a generalization of the classical positive definite kernels in meshfree methods.
Then, we construct and analyze kernel-based estimators by the kernel-based probability measures
in Section \ref{sec:MeshfreeApprox}.
In Theorems \ref{thm:kerappr-inter} and \ref{thm:kerappr-inter-noise}, we obtain the kernel-based estimators conditioned on non-noise data and noisy data mentioned in Examples \ref{exa:Hermite-Birkhoff} and \ref{exa:noise-Hermite-Birkhoff}, respectively.
We also investigate the error analysis of kernel-based estimators by the kernel-based probability measures such as the convergence in Theorems \ref{thm:kerappr-conv} and \ref{thm:kerappr-conv-noise} and the error bounds in Theorems \ref{thm:kerappr-error} and \ref{thm:kerappr-error-noise}.
This indicates that deterministic numerical problems could be solved by a stochastic approach.
In Section \ref{sec:MeshfreeApproxPDE}, we apply the kernel-based probability measures to construct the kernel-based estimators to approximate the solutions of elliptic partial differential equations.
This article gives the general theorems of interpolations by the kernel-based probability measures.
In the next article, we will investigate many research topics of meshfree methods by the kernel-based probability measures, for example, multiple-kernel approximation, optimal designs of kernel basis, and kernel-based solutions of stochastic differential equations.

\section{Kernel-based Probability Measures on Banach Spaces}\label{sec:KerProbSpace}

In this section, by the theorems of Gaussian measures on Banach spaces in \cite[Chapter VIII]{Stroock1993}, we generalize the original concept of kernel-based probability measures on Sobolev spaces in \cite{Ye2017PDE}.
We will investigate the kernel-based approximation methods of deterministic data by the kernel-based probability measures in Section \ref{sec:MeshfreeApprox}.
Roughly, the kernel-based probability measures can be viewed as a generalization of Wiener measures.

\begin{example}\label{exa:BrownianMotion}
Let $\Cont([0,1])$ be the collection of all real continuous functions on $[0,1]$ endowed with uniform norm.
Let $\Filter_{\Cont([0,1])}$ be the Borel $\sigma$-algebra of $\Cont([0,1])$.
The structure theorem of Wiener measures shows that there exists a probability measure $\PP_{W}$ on the measurable space $\left(\Cont([0,1]),\Filter_{\Cont([0,1])}\right)$
such that the coordinate mapping process $W_t(\omega):=\omega(t)$ for $\omega\in\Cont([0,1])$ is a standard Brownian motion on the probability space $\left(\Cont([0,1]),\Filter_{\Cont([0,1])},\PP_{W}\right)$.
The probability measure $\PP_{W}$ and the probability space $\left(\Cont([0,1]),\Filter_{\Cont([0,1])},\PP_{W}\right)$ are known as Wiener measure and Wiener space, respectively.
The Brownian motion $W$ is a Gaussian process with the mean $0$ and the covariance kernel $K(t,s):=\min\{t,s\}$.
To be more precise, for any $n\in\NN$ and any $t_1,\ldots,t_n\in[0,1]$, the random vector $W_T:=\left(W_{t_1},\ldots,W_{t_n}\right)^T$ has a multivariate normal distribution with the mean vector $0$ and the covariance matrix $\vA_{K,T}:=\big(K(t_j,t_k)\big)_{j,k=1}^{n,n}$, where $T:=\left\{t_1,\ldots,t_n\right\}$.
Moreover, the min kernel $K$ is a symmetric positive definite kernel. The normed space $\Cont([0,1])$ is a Banach space.
The dual space of $\Cont([0,1])$ is $\rca([0,1])$ which is the collection of all Radon measures on $[0,1]$ endowed with total variation, and the dual bilinear product on $\Cont([0,1])\times\rca([0,1])$ has the form
\[
\left\langle\omega,\nu\right\rangle_{\Cont([0,1])}
=
\int_0^1\omega(t)\nu(\ud t),
\quad
\text{for }\omega\in\Cont([0,1])\text{ and }
\nu\in\rca([0,1]).
\]
We also have that
\[
\Mean\big(\int_0^1W_t\nu_1(\ud t)\int_0^1W_s\nu_2(\ud s)\big)=
\int_0^1\int_0^1K(t,s)\nu_1(\ud t)\nu_2(\ud s),
\]
for $\nu_1,\nu_2\in\rca([0,1])$. More details are mentioned in \cite[Section 8.1]{Stroock1993}.
\end{example}

With the same idea of Wiener measures, we introduce kernel-based probability measures on Banach spaces.
Let $\Banach$ be a real Banach space which is a real complete normed vector space.
Let $\Banach^{\ast}$ be the dual space of $\Banach$ which consists of all real bounded linear functionals on $\Banach$.
We define the dual bilinear product $\langle\omega,L\rangle_{\Banach}:=L\omega$ for $\omega\in\Banach$ and $L\in\Banach^{\ast}$.
Given any finite many functionals $L_1,\ldots,L_n\in\Banach^{\ast}$, we define
a vector operator
\[
\vL_n\omega:=
\left(\langle\omega,L_1\rangle_{\Banach},\cdots,\langle\omega,L_n\rangle_{\Banach}\right)^T,\quad\text{for }\omega\in\Banach.
\]
Clearly, $\vL_n$ is a bounded linear operator from $\Banach$ into $\RR^n$.
Let $\Ker:\Banach^{\ast}\times\Banach^{\ast}\to\RR$ be a kernel.
Based on $\Ker$ and $\vL_n$, we define a squared matrix
\begin{equation}\label{eq:CovarianceMatrix}
\vA_{\Ker,\vL_n}:=
\begin{pmatrix}
\Ker(L_1,L_1)&\cdots&\Ker(L_1,L_n)\\
\vdots&\ddots&\vdots\\
\Ker(L_n,L_1)&\cdots&\Ker(L_n,L_n)
\end{pmatrix}
\in\RR^{n\times n}.
\end{equation}
Let $\Filter_{\Banach}$ be the Borel $\sigma$-algebra of $\Banach$ which is the collection of all Borel sets in $\Banach$.
Thus, the pair $\left(\Banach,\Filter_{\Banach}\right)$ can be viewed as a measurable space.
If the measurable space $\left(\Banach,\Filter_{\Banach}\right)$ is endowed with a probability measure $\PP$ such that $\left(\Banach,\Filter_{\Banach},\PP\right)$ becomes a probability space, then we call $\PP$ a Borel probability measure on $\Banach$. For any $L\in\Banach^{\ast}$, the random variable
\begin{equation}\label{eq:GaussianField}
S_{L}(\omega):=\langle\omega,L\rangle_{\Banach},\quad\text{for }\omega\in\Banach,
\end{equation}
is well-defined on the probability space $\left(\Banach,\Filter_{\Banach},\PP\right)$.
This shows that $S$ can be viewed as a stochastic process on $\left(\Banach,\Filter_{\Banach},\PP\right)$ indexed by functionals in $\Banach^{\ast}$.
Moreover,
the stochastic process $S$ is called a \emph{Gaussian process} on the probability space $\left(\Banach,\Filter_{\Banach},\PP\right)$ with the mean $\mu\in\Banach$ and the covariance kernel $\Ker:\Banach^{\ast}\times\Banach^{\ast}\to\RR$
if for any $n\in\NN$ and any $L_1,\ldots,L_n\in\Banach^{\ast}$,
the random vector
\[
\vS_{\vL_n}:=\left(S_{L_1},\ldots,S_{L_n}\right)^T,\text{ or }
\vS_{\vL_n}(\omega):=\vL_n\omega,\text{ for }\omega\in\Banach,
\]
has a multivariate normal distribution with the mean vector $\vL_n\mu$ and the covariance matrix $\vA_{\Ker,\vL_n}$.
The covariances of Gaussian processes can be viewed as generalization of the covariances of intrinsic random functions in \cite{Matheron1973}.
By the extension of covariance kernels on dual spaces, we generalize the definition of kernel-based probability measures in \cite[Definition 2]{Ye2017PDE}.

\begin{definition}\label{def:KerProbSpace}
A Borel probability measure $\PP_{\Ker}^{\mu}$ on a Banach space $\Banach$ is said a \emph{kernel-based probability measure} on $\Banach$ with a mean $\mu\in\Banach$ and a covariance kernel $\Ker:\Banach^{\ast}\times\Banach^{\ast}\to\RR$
if the stochastic process $S$ in equation \eqref{eq:GaussianField} is a Gaussian process on
the probability space $\big(\Banach,\Filter_{\Banach},\PP_{\Ker}^{\mu}\big)$
with the mean $\mu$ and the covariance kernel $\Ker$.
The measurable space $\left(\Banach,\Filter_{\Banach}\right)$ endowed with $\PP_{\Ker}^{\mu}$ is said a \emph{kernel-based probability space} $\big(\Banach,\Filter_{\Banach},\PP_{\Ker}^{\mu}\big)$.
\end{definition}

Since $S$ is a Gaussian process on $\Banach$, then $\PP_{\Ker}^{\mu}$ is a Gaussian measure on $\Banach$.
If $\mu=0$, then $\PP_{\Ker}^{\mu}$ is a centered Gaussian measure.
If $\Var\left(S_L\right)=\Cov(S_L,S_L)=\Ker\left(L,L\right)>0$ for all $L\in\Banach^{\ast}\backslash\{0\}$, then $\PP_{\Ker}^{\mu}$ is a non-degenerate Gaussian measure.
The kernel-based probability measure $\PP_{\Ker}^{\mu}$ is a non-degenerate centered Gaussian measure if and only if
$\mu=0$ and $\Mean\abs{S_L}^2=\Ker\left(L,L\right)>0$ for all $L\in\Banach^{\ast}\backslash\{0\}$.
The normal density functions from the Gaussian measures will be used to compute the kernel-based estimators.
In this article, for the generalized interpolations, we focus on the constructions of kernel-based estimators by the covariance kernels.
Therefore, the special Gaussian measures are recalled the kernel-based probability measures to avoid the confusion of the different research areas such as approximation theory and probability theory.

\begin{proposition}\label{pro:PDK}
If $\Ker$ is a covariance kernel of a kernel-based probability measure $\PP_{\Ker}^{\mu}$ on a Banach space $\Banach$, then for any $n\in\NN$ and any $L_1,\ldots,L_n\in\Banach^{\ast}$, the matrix $\vA_{\Ker,\vL_n}$ is a symmetric positive definite matrix.
\end{proposition}
\begin{proof}
Based on Definition \ref{def:KerProbSpace}, the matrix $\vA_{\Ker,\vL_n}$ is the covariance matrix of the multivariate normal vector $\vS_{\vL_n}$. Since each covariance matrix is a symmetric positive definite matrix, the property of $\vA_{\Ker,\vL_n}$ can be proved.
\end{proof}

Let $K$ be a positive definite kernel on a domain $\Domain\subseteq\Rd$. Thus, for any $n\in\NN$ and any $\vx_1,\ldots,\vx_n\in\Domain$, the matrix $\vA_{K,X}:=\big(K(\vx_j,\vx_k)\big)_{j,k=1}^{n,n}$ is a symmetric positive definite matrix, where $X:=\left\{\vx_1,\ldots,\vx_n\right\}$. By \cite[Theorem~10.10]{Wendland2005}, there exists the unique reproducing kernel Hilbert space $\Hilbert_K(\Domain)$ with the reproducing kernel $K$. By the reproducing properties, the point evaluation functions $\delta_{\vx},\delta_{\vy}\in\Hilbert_K(\Domain)^{\ast}$ and the inner product $\big(\delta_{\vx},\delta_{\vy}\big)_{\Hilbert_K(\Domain)^{\ast}}=K(\vx,\vy)$ for $\vx,\vy\in\Domain$.
Thus, $\vA_{K,X}=\big(\big(\delta_{\vx_j},\delta_{\vx_k}\big)_{\Hilbert_K(\Domain)^{\ast}}\big)_{j,k=1}^{n,n}$.
Moreover, \cite[Theorem~16.7]{Wendland2005} shows that $\left(L_1,L_2\right)_{\Hilbert_K(\Domain)^{\ast}}=L_{1,\vx}L_{2,\vy}K(\vx,\vy)$ for $L_1,L_2\in\Hilbert_K(\Domain)^{\ast}$.
Thus, the kernel $\Ker(L_1,L_2):=L_{1,\vx}L_{2,\vy}K(\vx,\vy)$ is well-defined on $\Hilbert_K(\Domain)^{\ast}$.
This shows that the covariance kernels of kernel-based probability measures can be viewed as a
generalization of the classical positive definite kernels in meshfree methods.

In Example \ref{exa:BrownianMotion}, the classical Wiener measure is a kernel-based probability measure on $\Cont([0,1])$
with the mean $0$ and the covariance kernel
\[
\Ker(\nu_1,\nu_2):=\int_0^1\int_0^1\min\left\{t,s\right\}\nu_1(\ud t)\nu_2(\ud s),
\quad\text{for }\nu_1,\nu_2\in\rca([0,1]).
\]
We will give another example of kernel-based probability measures on Sobolev spaces in
\cite{Ye2016PDK,Ye2017Kriging,Ye2017PDE}.

%
%

\begin{example}\label{exa:PDK-Sobolev}
Suppose that $\Domain\subseteq\Rd$ is a regular and compact domain.
Let $\Hilbert^m(\Domain)$ be the real $\Leb_2$-based Sobolev space of the order $m>d/2$,
that is,
\[
\Hilbert^m(\Domain):=\left\{
\omega\in\Leb_2(\Domain)
:D^{\valpha}\omega\in\Leb_2(\Domain)\text{ for }\valpha\in\NN_0^d\text{ and }\abs{\valpha}\leq m
\right\},
\]
where $D^{\valpha}$ is the partial derivative of the order $\valpha$.
Suppose that $\mu\in\Hilbert^m(\Domain)$.
Let $K$ be a symmetric strictly positive definite kernel on $\Domain$.
Suppose that $K\in\Cont^{2m}(\Domain\times\Domain)$ and $K$ has the $2m$th partial derivatives with Lipschitz conditions.
By \cite[Theorem~6.1]{Ye2016PDK} and \cite[Theorem 1]{Ye2017PDE},
there exists a kernel-based probability measure $\PP_{\Ker}^{\mu}$ on $\Hilbert^m(\Domain)$
with the mean $\mu$ and the covariance kernel
\[
\Ker(L_1,L_2):=L_{1,\vx}L_{2,\vy}K(\vx,\vy),
\quad\text{for }L_1,L_2\in\Hilbert^m(\Domain)^{\ast},
\]
where the notations $L_{\vx}$ and $L_{\vy}$
represent an operator $L$ acting on the first argument $\vx$ and
the second argument $\vy$ of the kernel $K(\vx,\vy)$, respectively.
Specially, we define a linear functional
\[
\vartheta^{\valpha}(\omega):=
\int_{\Domain}\left(D^{\valpha}\omega\right)(\vx)\varphi(\vx)\ud\vx,
\quad\text{for }\omega\in\Hilbert^m(\Domain),
\]
where $\abs{\valpha}\leq m$ and $\varphi\in\Cont_0^{\infty}(\Domain)$.
Thus, we have that
$\vartheta^{\valpha}\in\Hilbert^m(\Domain)^{\ast}$.
This shows that
\[
\Ker\big(\vartheta^{\valpha_1},\vartheta^{\valpha_2}\big)=
\int_{\Domain}\int_{\Domain}D_{\vx}^{\valpha_1}D_{\vy}^{\valpha_2}K(\vx,\vy)\varphi(\vx)\varphi(\vy)\ud\vx\ud\vy,
\]
for $\abs{\valpha_1},\abs{\valpha_2}\leq m$.
Moreover, Sobolev imbedding theorem guarantees that $\delta_{\vx}\circ D^{\valpha}\in\Hilbert^m(\Domain)^{\ast}$ for $\vx\in\Domain$ and $\abs{\valpha}<m-d/2$.
This shows that
\[
\Ker\big(\delta_{\vz_1}\circ D^{\valpha_1},\delta_{\vz_2}\circ D^{\valpha_2}\big)=
D_{\vx}^{\valpha_1}D_{\vy}^{\valpha_2}K(\vx,\vy)|_{\vx=\vz_1,\vy=\vz_2},
\]
for $\vz_1,\vz_2\in\Domain$ and $\abs{\valpha_1},\abs{\valpha_2}<m-d/2$.
\end{example}

We already know that the reproducing kernel Hilbert spaces of Sobolev-spline kernels are equivalent to the Sobolev space.
For a general reproducing kernel $K\in\Hilbert^{2m}(\Domain\times\Domain)$ for $m>d/2$, \cite[Theorem 3.1]{CialencoFasshauerYe2012} assures that there exists the kernel-based probability measure $\PP_{\Ker}^{\mu}$ on $\Hilbert_K(\Domain)$ with the mean $\mu\in\Hilbert_K(\Domain)$ and the covariance kernel
\[
\Ker\left(L_1,L_2\right):=\int_{\Domain}L_{1,\vx}K(\vx,\vt)L_{2,\vy}K(\vy,\vt)\ud\vt,\quad\text{for }L_1,L_2\in\Hilbert_K(\Domain)^{\ast}.
\]
Thus, the kernel-based probability measures on Banach spaces in Definition \ref{def:KerProbSpace} can be viewed as a generalization of original kernel-based probability measures on reproducing kernel Hilbert spaces
in \cite{CialencoFasshauerYe2012,FasshauerYe2012SPDEMCQMC,FasshauerYe2012Bonn}
or Sobolev spaces in \cite{Ye2016PDK,Ye2017Kriging,Ye2017PDE}.

\begin{proposition}\label{pro:SL-L}
If $\PP_{\Ker}^{\mu}$ is a kernel-based probability measure on a Banach space $\Banach$,
then for any $L\in\Banach^{\ast}$, the random variable $S_L$ satisfies that
\[
\Mean\abs{S_L}^2\leq C\norm{L}_{\Banach^{\ast}}^2,
\]
where the constant $C\geq0$ is independent of $L$.
\end{proposition}
\begin{proof}
Based on the Fernique's theorem, we have the nonnegative constant
\[
C:=\int_{\Banach}\norm{\omega}_{\Banach}^2\PP_{\Ker}^{\mu}(\ud\omega)<\infty.
\]
Equation \eqref{eq:GaussianField} shows that
\begin{equation}\label{eq:SL-L-1}
\abs{S_L(\omega)}^2
=\abs{\langle\omega,L\rangle_{\Banach}}^2\leq\norm{L}_{\Banach^{\ast}}^2\norm{\omega}_{\Banach}^2,
\quad
\text{for all }\omega\in\Banach.
\end{equation}
Integrating the both sides of equation \eqref{eq:SL-L-1}, we complete the proof.
\end{proof}

Let $\Leb_2(\Banach,\Filter_{\Banach},\PP_{\Ker}^{\mu})$ be a collection of all random variables on the kernel-based probability space $\big(\Banach,\Filter_{\Banach},\PP_{\Ker}^{\mu}\big)$ with finite second moments.
Since $\Mean\abs{S_L}^2=\left(L\mu\right)^2+\Ker(L,L)<\infty$ for $L\in\Banach^{\ast}$, we have that
$S_L\in\Leb_2(\Banach,\Filter_{\Banach},\PP_{\Ker}^{\mu})$.
Let $\Hilbert_S$ be the competition of the linear space $\Span\left\{S_L:L\in\Banach^{\ast}\right\}$
by the second-moment norm.
Clearly $\Hilbert_S$ is a closed subspace of $\Leb_2(\Banach,\Filter_{\Banach},\PP_{\Ker}^{\mu})$.
Thus, $\Hilbert_S$ is a Hilbert space endowed with the inner product
\[
\left(U,V\right)_{\Hilbert_S}=\Mean\left(UV\right),\quad
\text{for }U,V\in\Hilbert_S.
\]

\begin{proposition}\label{pro:SL-L-imbedding}
If $\Ker(L,L)>0$ for all $L\in\Banach^{\ast}\backslash\{0\}$, then the dual space $\Banach^{\ast}$ is imbedded in the Hilbert space $\Hilbert_S$.
\end{proposition}
\begin{proof}
By equation \eqref{eq:GaussianField},
the Gaussian process $S$ can be viewed as a linear map from $\Banach^{\ast}$ into $\Hilbert_S$.
By Proposition \ref{pro:SL-L}, the map $S$ is continuous.

If we verify that $S$ is a bijective map, then the proof is complete.
We take any $L_1,L_2\in\Banach^{\ast}$ such that $L_1\neq L_2$.
Let $L:=L_1-L_2$. Then, $L\neq 0$. So, $\Mean\abs{S_{L_1}-S_{L_2}}^2=\Mean\abs{S_L}^2=\left(L\mu\right)^2+\Ker(L,L)>0$.
This shows that $S_{L_1}\neq S_{L_2}$. Therefore, the map $S$ is bijective.
\end{proof}

In this article, we only look at the basic theorems of kernel-based probability measures.
Another theorems of kernel-based probability measures can be obtained by the same methods as shown in
\cite[Chapter 4]{BerlinetThomas2004} and \cite[Chapter VIII]{Stroock1993}.
The kernel-based probability measure will be viewed as a numerical tool to renew the kernel-based approximation methods as follows.

\section{Generalized Interpolations by Kernel-based Probability Measures}\label{sec:MeshfreeApprox}

In this section, we show kernel-based approximation methods by kernel-based probability measures.
To be more precise, we construct and analyze the kernel-based estimators conditioned on the given data by the kernel-based probability measures.

In approximation problems, for an unknown element $u$ in the Banach space $\Banach$ and a given functional $L$ in the dual space $\Banach^{\ast}$,
we want to estimate the value $Lu$ conditioned on the given data evaluated by $u$.
The element $u$ can be viewed as an exact solution in $\Banach$. But the exact value $Lu$ is always unknown or uncomputable.
Thus, we need to approximate $Lu$ with a numerical method.

Let $\PP_{\Ker}^{\mu}$ be a kernel-based probability measure on $\Banach$ with a mean $\mu\in\Banach$
and a covariance kernel $\Ker:\Banach^{\ast}\times\Banach^{\ast}\to\RR$ in Definition \ref{def:KerProbSpace}.
The kernel-based probability measure $\PP_{\Ker}^{\mu}$
provides a numerical tool to compute kernel-based estimators.
To be more precise, we will construct the kernel-based estimators to approximate $Lu$ by the covariance kernel $\Ker$.
For convenience, the mean $\mu$ and the covariance kernel $\Ker$ are always fixed in this section.

We look at the collection $\Aset$ of all elements in $\Banach$ satisfying the generalized interpolation conditions.
Roughly, $\Aset$ is a certain set of ``conditions'' known to occur.
We call the subset $\Aset$ an \emph{interpolation event}.
We always think that $u\in\Aset$.
The kernel-based estimators will be computed from the average over $\Aset$ measured by $\PP_{\Ker}^{\mu}$.
Specially, we illustrate two examples of $\Aset$ based on non-noise data or noisy data.

\begin{example}\label{exa:Hermite-Birkhoff}
The non-noise data
\[
(L_1,f_1),\ldots,(L_n,f_n)\in\Banach^{\ast}\times\RR,
\]
are evaluated by the element $u\in\Banach$, that is,
\[
\langle u,L_1\rangle_{\Banach}=f_1,\ldots,\langle u,L_n\rangle_{\Banach}=f_n.
\]
For the interpolation, we define the interpolation event
\[
\Aset_n:=\left\{\omega\in\Banach:\vL_n\omega=\vf_n\right\},
\]
where $\vL_n:=\left(L_1,\cdots,L_n\right)^T$ and $\vf_n:=\left(f_1,\cdots,f_n\right)^T$.
It is obvious that $u\in\Aset_n$ and $\Aset_n$ includes all elements in $\Banach$ satisfying the generalized interpolation conditions.
According to equation \eqref{eq:GaussianField}, $\Aset_n$ can be rewritten by the Gaussian process $S$, that is,
\begin{equation}\label{eq:Hermite-Birkhoff}
\Aset_n=\left\{\omega\in\Banach:\vS_{\vL_n}(\omega)=\vf_n\right\}.
\end{equation}
Since $\vL_n$ is linear and continuous on $\Banach$, the set $\Aset_n$ is a closed affine set of $\Banach$.
Thus $\Aset_n$ belongs to the Borel $\sigma$-algebra $\Filter_{\Banach}$.
\end{example}

Let $B_n(\vz,r):=\left\{\vv\in\RR^n:\norm{\vv-\vz}_2\leq r\right\}$ be the closed ball centered at $\vz\in\RR^n$ with the radius $r>0$. Thus, we have that $B_n(\vz,r)=\vz+B_n(0,r)$.

\begin{example}\label{exa:noise-Hermite-Birkhoff}
The noisy data
\[
\big(L_1,\hat{f}_1\big),\ldots,\big(L_n,\hat{f}_n\big)\in\Banach^{\ast}\times\RR,
\]
are evaluated by the element $u\in\Banach$ for the noise margin $\epsilon_n>0$, that is,
\[
\langle u,L_1\rangle_{\Banach}=\hat{f}_1+\xi_1,\ldots,
\langle u,L_n\rangle_{\Banach}=\hat{f}_n+\xi_n,
\]
where
\[
\xi_1^2+\cdots+\xi_n^2\leq\epsilon_n^2.
\]
To fit the noisy data, we define the interpolation event
\[
\Aset_n^{\epsilon_n}:=\left\{\omega\in\Banach:\vL_n\omega\in B_n\big(\hatvf_n,\epsilon_n\big)\right\},
\]
where $\vL_n:=\left(L_1,\cdots,L_n\right)^T$ and
$\hatvf_n:=\big(\hat{f}_1,\cdots,\hat{f}_n\big)^T$.
We observe that $u\in\Aset_n^{\epsilon_n}$ and $\Aset_n^{\epsilon_n}$ includes all elements in $\Banach$ interpolating the noisy data in the confidence region.
Comparing with Example \ref{exa:Hermite-Birkhoff}, we find that
$\norm{\vf_n-\hatvf_n}_2\leq\epsilon_n$.
This shows that $\Aset_n\subseteq\Aset_n^{\epsilon_n}$ and $\Aset_n^{\epsilon_n}$ is closed to $\Aset_n$ when $\epsilon_n\to0$.
Moreover, $\Aset_n^{\epsilon_n}$ can be rewritten as
\begin{equation}\label{eq:noise-Hermite-Birkhoff}
\Aset_n^{\epsilon_n}=\left\{\omega\in\Banach:\vS_{\vL_n}(\omega)\in B_n\big(\hatvf_n,\epsilon_n\big)\right\}.
\end{equation}
Since $B_n\big(\hatvf_n,\epsilon_n\big)$ is a closed ball of $\RR^n$, we have that $\Aset_n^{\epsilon_n}\in\Filter_{\Banach}$.
\end{example}

In the beginning, we illustrate the notations.
Given finite many $L_1,\ldots,L_n\in\Banach^{\ast}$, we let $\vL_n:=\left(L_1,\cdots,L_n\right)^T$.
All non-noise data values $\vf_n$ or noisy data values $\hatvf_n$ are evaluated by the exact solution $u\in\Banach$ acting on $\vL_n$ same as Examples \ref{exa:Hermite-Birkhoff} or \ref{exa:noise-Hermite-Birkhoff}, that is, $\vf_n:=\vL_nu$ or $\hatvf_n:=\vL_nu+\vxi_n$ where $\vxi_n\in B_n\big(0,\epsilon_n\big)$.
Moreover, we define a map $\vrho_n$ from $\cup_{m\geq n}\RR^m$ into $\RR^n$ as
\[
\vrho_n(\vzeta):=\vz,\quad
\text{for }\vzeta:=\begin{pmatrix}\vz\\\ve\end{pmatrix}\in\RR^m,~\vz\in\RR^n,~\ve\in\RR^{m-n},\text{ and }m\geq n.
\]
This shows that $\vf_n=\rho_{n}\big(\vf_{n+1}\big)$.
But $\hatvf_{n}$ may not be equal to $\rho_{n}\big(\hatvf_{n+1}\big)$.

The Gaussian process $S$ is defined on $\big(\Banach,\Filter_{\Banach},\PP_{\Ker}^{\mu}\big)$ in equation \eqref{eq:GaussianField}.
Thus, we have the multivariate normal random variables $S_L$ and $\vS_{\vL_n}$.
Let $p_{L,\vL_n}(z,\vz)$ and $p_{\vL_n}(\vz)$ be the joint probability density functions of $\left(S_L,\vS_{\vL_n}\right)$ and $\vS_{\vL_n}$, respectively.
Let $p_{L|\vL_n}(z|\vz)$ be the conditional probability density function of $S_L$ given $\vS_{\vL_n}$.
We define a vector function
\[
\vb_{\Ker,\vL_n}(L):=
\left(\Ker(L,L_1),\cdots,\Ker(L,L_n)\right)^T,
\quad
\text{for }L\in\Banach^{\ast}.
\]
This shows that $\vb_{\Ker,\vL_n}$ is a map from $\Banach^{\ast}$ into $\RR^n$ and
$\vb_{\Ker,\vL_n}(L)=\Mean(S_L\vS_{\vL_n})-\Mean(S_L)\Mean(\vS_{\vL_n})$.
Clearly $\vb_{\Ker,\vL_n}$ composes of the kernel basis $\Ker(\cdot,L_1),\ldots,\Ker(\cdot,L_n)$.
The covariance matrix $\vA_{\Ker,\vL_n}$ of $\vS_{\vL_n}$ is defined in equation \eqref{eq:CovarianceMatrix}.
Proposition~\ref{pro:PDK} shows that $\vA_{\Ker,\vL_n}$ is a symmetric positive definite matrix.
This assures that the \emph{pseudo inverse} $\vA_{\Ker,\vL_n}^\dag$ of $\vA_{\Ker,\vL_n}$ is well-defined, that is,
\[
\vA_{\Ker,\vL_n}\vA_{\Ker,\vL_n}^{\dag}\vA_{\Ker,\vL_n}=\vA_{\Ker,\vL_n}\text{ and }
\vA_{\Ker,\vL_n}^{\dag}\vA_{\Ker,\vL_n}\vA_{\Ker,\vL_n}^{\dag}=\vA_{\Ker,\vL_n}^{\dag}.
\]
If $\vA_{\Ker,\vL_n}=0$, then $\vA_{\Ker,\vL_n}^{\dag}=0$. Since the kernel-based estimators are mainly dependent of $\vA_{\Ker,\vL_n}$ as follows, the algorithms and theorems of kernel-based estimators will be trivial when $\vA_{\Ker,\vL_n}=0$. To simplify the discussions, we suppose that $\vA_{\Ker,\vL_n}$ is not equal to $0$ in this section.
Let
$\lambda_{\min}\big(\vA_{\Ker,\vL_n}\big)$ be the smallest positive eigenvalue of $\vA_{\Ker,\vL_n}$.

\subsection{Representations of Kernel-based Estimators}\label{sec:KerAppr}

Given the interpolation event $\Aset\in\Filter_{\Banach}$, we know that the unknown solution $u\in\Aset$.
Each $\omega\in\Aset$ could be a solution. To avoid missing any element in $\Aset$, the estimator at the fixed $L\in\Banach^{\ast}$ is evaluated by the average over $\Aset$.
The kernel-based probability measure $\PP_{\Ker}^{\mu}$
provides a numerical tool to compute the average to estimate the value $Lu$.
When $\PP_{\Ker}^{\mu}(\Aset)>0$, the average can be written as
\[
\frac{1}{\PP_{\Ker}^{\mu}(\Aset)}\int_{\Aset}\langle\omega,L\rangle_{\Banach}\PP_{\Ker}^{\mu}(\ud\omega)
=\int_{\Banach}\langle\omega,L\rangle_{\Banach}\PP_{\Ker}^{\mu}(\ud\omega|\Aset).
\]
Therefore,
the \emph{kernel-based estimator} $s_{\Aset}(L)$ at $L$ conditioned on $\Aset$ is defined by
\begin{equation}\label{eq:kerappr-1}
s_{\Aset}(L):=\int_{\Banach}\langle\omega,L\rangle_{\Banach}\PP_{\Ker}^{\mu}(\ud\omega|\Aset).
\end{equation}
Clearly, $s_{\Aset}(L)$ is still well-defined when $\PP_{\Ker}^{\mu}(\Aset)=0$.
Using the Gaussian process $S$ in equation \eqref{eq:GaussianField}, the kernel-based estimator $s_{\Aset}(L)$ can be rewritten by the conditional mean.

\begin{proposition}\label{pro:kerappr}
If the interpolation event $\Aset\in\Filter_{\Banach}$, then for any $L\in\Banach^{\ast}$, the kernel-based estimator $s_{\Aset}(L)$ can be written as
\begin{equation}\label{eq:kerappr-2}
s_{\Aset}(L)=\Mean\left(S_L|\Aset\right).
\end{equation}
\end{proposition}
\begin{proof}
Since $S_{L}(\omega)=\langle \omega,L\rangle_{\Banach}$ for $\omega\in\Banach$, equation \eqref{eq:kerappr-1} can be rewritten as
\[
s_{\Aset}(L)=\int_{\Banach}S_L(\omega)\PP_{\Ker}^{\mu}(\ud\omega|\Aset)
=\Mean\left(S_L|\Aset\right).
\]
\end{proof}

Next, we look at explicit formulas of kernel-based estimators conditioned on the special interpolation events $\Aset_n$ and $\Aset_n^{\epsilon_n}$ in Examples \ref{exa:Hermite-Birkhoff} and \ref{exa:noise-Hermite-Birkhoff}, respectively.

\begin{theorem}\label{thm:kerappr-inter}
If the interpolation event $\Aset_n$ based on the non-noise data $(\vL_n,\vf_n)$ is the same as Example \ref{exa:Hermite-Birkhoff}, then for any $L\in\Banach^{\ast}$, the kernel-based estimator $s_{\Aset_n}(L)$
can be written as
\begin{equation}\label{eq:kerappr-3}
s_{\Aset_n}(L)=L\mu+\vb_{\Ker,\vL_n}(L)^T\vA_{\Ker,\vL_n}^{\dag}(\vf_n-\vL_n\mu).
\end{equation}
\end{theorem}
\begin{proof}
Based on equation \eqref{eq:Hermite-Birkhoff}, we have that
\begin{equation}\label{eq:kerappr-inter-1}
\Mean\left(S_L|\Aset_n\right)=\Mean\left(S_L|\vS_{\vL_n}=\vf_n\right).
\end{equation}
Since the Gaussian process $S$ has the mean $\mu$ and the covariance kernel $\Ker$, the random variables
$\left(S_L,\vS_{\vL_n}\right)$ have a joint normal distribution with the mean vector
\[
\begin{pmatrix}
L\mu\\\vL_n\mu
\end{pmatrix},
\]
and the covariance matrix
\[
\begin{pmatrix}
\Ker(L,L)&\vb_{\Ker,\vL_n}(L)^T\\
\vb_{\Ker,\vL_n}(L)&\vA_{\Ker,\vL_n}
\end{pmatrix}.
\]
This shows that the conditional normal density function $p_{L|\vL_n}$ can be written as
\begin{equation}\label{eq:condpdf}
p_{L|\vL_n}(z|\vz)=
\frac{1}{\sqrt{2\pi}\sigma_{L|\vL_n}}\exp\left(-\frac{\left(z-m_{L|\vL_n}(\vz)\right)^2}{2\sigma_{L|\vL_n}^2}\right),
\quad\text{for }z\in\RR\text{ and }\vz\in\RR^n,
\end{equation}
where the mean
\begin{equation}\label{eq:mean}
m_{L|\vL_n}(\vz):=L\mu+\vb_{\Ker,\vL_n}(L)^T\vA_{\Ker,\vL_n}^{\dag}(\vz-\vL_n\mu),
\end{equation}
and the variance
\begin{equation}\label{eq:var}
\sigma_{L|\vL_n}^2:=\Ker(L,L)-\vb_{\Ker,\vL_n}(L)^T\vA_{\Ker,\vL_n}^{\dag}\vb_{\Ker,\vL_n}(L).
\end{equation}
Thus, the conditional mean $\Mean\left(S_L|\vS_{\vL_n}=\vf_n\right)$ can be computed by $p_{L|\vL_n}$, that is,
\begin{equation}\label{eq:kerappr-inter-2}
\Mean\left(S_L|\vS_{\vL_n}=\vf_n\right)=\int_{\RR}vp_{L|\vL_n}(z|\vf_n)\ud z
=m_{L|\vL_n}(\vf_n).
\end{equation}
Putting equations \eqref{eq:kerappr-inter-1} and \eqref{eq:kerappr-inter-2} into equation \eqref{eq:kerappr-2},
we have that
\[
s_{\Aset_n}(L)=m_{L|\vL_n}(\vf_n).
\]
\end{proof}

Let
\[
\psi_{\vL_n}(\vz):=
\exp\left(-\frac{1}{2}\left(\vz-\vL_n\mu\right)^T\vA_{\Ker,\vL_n}^{\dag}\left(\vz-\vL_n\mu\right)\right),
\quad
\text{for }\vz\in\RR^n.
\]
Thus,
the normal density function $p_{\vL_n}$ can be written as
\[
p_{\vL_n}(\vz)=\frac{\psi_{\vL_n}(\vz)}{(2\pi)^{n/2}\sqrt{\Det^{\dag}\big(\vA_{\Ker,\vL_n}\big)}},
\quad\text{for }\vz\in\RR^n,
\]
where $\Det^{\dag}$ is the pseudo determinant.
Let $I_{B}$ be the indicator function of the subset $B$, that is,
$I_{B}(\vz)=1$ if $\vz\in B$ otherwise $I_{B}(\vz)=0$.
We define a vector function
\[
\veta_{\Ker,\vL_n}^{\epsilon_n}(\vz):=
\frac{\int_{B_n(\vz,\epsilon_n)}\vv\psi_{\vL_n}(\vv)\ud\vv}
{\int_{B_n(\vz,\epsilon_n)}\psi_{\vL_n}(\vv)\ud\vv},
\quad\text{for }\vz\in\RR^n.
\]
Thus, we have that
\begin{equation}\label{eq:eta-vec}
\veta_{\Ker,\vL_n}^{\epsilon_n}(\vz)=
\frac{\int_{\RR^n}\vv p_{\vL_n}(\vv)I_{B_n(0,\epsilon_n)}(\vz-\vv)\ud\vv}
{\int_{\RR^n}p_{\vL_n}(\vv)I_{B_n(0,\epsilon_n)}(\vz-\vv)\ud\vv}.
\end{equation}
This shows that $\veta_{\Ker,\vL_n}^{\epsilon_n}\big(\hatvf_n\big)$
is the truncated mean of the multivariate normal vector $\vS_{\vL_n}$ conditioned on $\vS_{\vL_n}\in B_n\big(\hatvf_n,\epsilon_n\big)$, that is,
\[
\veta_{\Ker,\vL_n}^{\epsilon_n}\big(\hatvf_n\big)=
\Mean\left(\vS_{\vL_n}\big|\vS_{\vL_n}\in B_n\big(\hatvf_n,\epsilon_n\big)\right)
=
\int_{\Aset_n^{\epsilon_n}}\vL_n\omega\PP_{\Ker}^{\mu}(\ud\omega).
\]
Thus, we know that $\veta_{\Ker,\vL_n}^{\epsilon_n}\big(\hatvf_n\big)$
is the average over $B_n\big(\hatvf_n,\epsilon_n\big)$ measured by $\PP_{\Ker}^{\mu}$.
Therefore, we have that
\[
\veta_{\Ker,\vL_n}^{\epsilon_n}\big(\hatvf_n\big)\in B_n\big(\hatvf_n,\epsilon_n\big),
\]
and
\[
\lim_{\epsilon_n\to0}\veta_{\Ker,\vL_n}^{\epsilon_n}\big(\hatvf_n\big)=\vf_n.
\]

\begin{theorem}\label{thm:kerappr-inter-noise}
If the interpolation event $\Aset_n^{\epsilon_n}$ based on the noisy data $(\vL_n,\hatvf_n)$ and the noise margin $\epsilon_n$ is the same as Example \ref{exa:noise-Hermite-Birkhoff}, then for any $L\in\Banach^{\ast}$, the kernel-based estimator $s_{\Aset_n^{\epsilon_n}}(L)$ can be written as
\begin{equation}\label{eq:kerappr-4}
s_{\Aset_n^{\epsilon_n}}(L)=L\mu+\vb_{\Ker,\vL_n}(L)^T\vA_{\Ker,\vL_n}^{\dag}
\left(\veta_{\Ker,\vL_n}^{\epsilon_n}\big(\hatvf_n\big)-\vL_n\mu\right).
\end{equation}
\end{theorem}
\begin{proof}
Based on equation \eqref{eq:noise-Hermite-Birkhoff}, we have that
\begin{equation}\label{eq:kerappr-inter-noise-1}
s_{\Aset_n^{\epsilon_n}}(L)=
\Mean\left(S_L|\Aset_n^{\epsilon_n}\right)
=\Mean\left(S_L\big|\vS_{\vL_n}\in B_n\big(\hatvf_n,\epsilon_n\big)\right).
\end{equation}
The normal distributions of $S_L$ and $\vS_{\vL_n}$ show that
\begin{equation}\label{eq:kerappr-inter-noise-2}
\Mean\left(S_L\big|\vS_{\vL_n}\in B_n\big(\hatvf_n,\epsilon_n\big)\right)
=
\frac{\int_{\RR}\int_{B_n(\hatvf_n,\epsilon_n)}zp_{L,\vL_n}(z,\vv)\ud\vv\ud z}
{\int_{B_n(\hatvf_n,\epsilon_n)}p_{\vL_n}(\vv)\ud\vv}
\end{equation}
Replacing $\vz=\hatvf_n$ in equation
\eqref{eq:eta-vec}, we have that
\begin{equation}\label{eq:kerappr-inter-noise-3}
\veta_{\Ker,\vL_n}^{\epsilon_n}\big(\hatvf_n\big)=
\frac{\int_{\RR^n}\vv p_{\vL_n}(\vv)I_{B_n(0,\epsilon_n)}\big(\hatvf_n-\vv\big)\ud\vv}
{\int_{\RR^n}p_{\vL_n}(\vv)I_{B_n(0,\epsilon_n)}\big(\hatvf_n-\vv\big)\ud\vv}.
\end{equation}
Since $p_{L,\vL_n}\left(z,\vv\right)=p_{L|\vL_n}\left(z|\vv\right)p_{\vL_n}(\vv)$,
we know that
\begin{equation}\label{eq:kerappr-inter-noise-4}
\int_{\RR}zp_{L,\vL_n}(z,\vv)\ud z
=
p_{\vL_n}(\vv)\int_{\RR}zp_{L|\vL_n}(z|\vv)\ud z
=
m_{L|\vL_n}(\vv)p_{\vL_n}(\vv).
\end{equation}
Integrating the both sides of equation \eqref{eq:kerappr-inter-noise-4} with respect to $\vv$,
we have that
\begin{equation}\label{eq:kerappr-inter-noise-5}
\begin{split}
&\int_{B_n(\hatvf_n,\epsilon_n)}
\int_{\RR}zp_{L,\vL_n}(z,\vv)\ud z\ud\vv
=
\int_{\RR^n}\int_{\RR}zp_{L,\vL_n}(z,\vv)I_{B_n(0,\epsilon_n)}\big(\hatvf_n-\vv\big)
\ud v\ud\vv
\\
=&\int_{\RR^n}\left(L\mu+\vb_{\Ker,\vL_n}(L)^T\vA_{\Ker,\vL_n}^{\dag}\left(\vv-\vL_n\mu
\right)\right)p_{\vL_n}(\vv)I_{B_n(0,\epsilon_n)}\big(\hatvf_n-\vv\big)\ud\vv
\end{split}
\end{equation}
Putting equations \eqref{eq:kerappr-inter-noise-3} and \eqref{eq:kerappr-inter-noise-5} into equation \eqref{eq:kerappr-inter-noise-2}, we compute that the conditional mean
\begin{equation}\label{eq:kerappr-inter-noise-6}
\Mean\left(S_L\big|\vS_{\vL_n}\in B_n\big(\hatvf_n,\epsilon_n\big)\right)
=
L\mu+\vb_{\Ker,\vL_n}(L)^T\vA_{\Ker,\vL_n}^{\dag}
\left(\veta_{\Ker,\vL_n}^{\epsilon_n}\big(\hatvf_n\big)-\vL_n\mu\right).
\end{equation}
Combining equations \eqref{eq:kerappr-inter-noise-1} and \eqref{eq:kerappr-inter-noise-6}, we complete the proof.
\end{proof}

Based on Theorems \ref{thm:kerappr-inter} and \ref{thm:kerappr-inter-noise},
we give the definition of kernel-based estimators conditioned on the given data.

\begin{definition}\label{def:kerapprox}
The estimator $s_{\Aset_n}(L)$ in equation \eqref{eq:kerappr-3} is called a kernel-based estimator at the functional $L$
conditioned on the non-noise data $(\vL_n,\vf_n)$.
The estimator $s_{\Aset_n^{\epsilon_n}}(L)$ in equation \eqref{eq:kerappr-4} is called a kernel-based estimator
at the functional $L$ conditioned on the noisy data $\big(\vL_n,\hatvf_n\big)$ and the noise margin $\epsilon_n$.
Specially, we rewrite $s_{\Aset_n}(L)$ and $s_{\Aset_n^{\epsilon_n}}(L)$ as $s_{\vL_n,\vf_n}(L)$
and $s_{\vL_n,\hatvf_n,\epsilon_n}(L)$
\footnote{In this article, the mean $\mu$ and the kernel $\Ker$ are always fixed in computation of $s_{\vL_n,\vf_n}(L)$
and $s_{\vL_n,\hatvf_n,\epsilon_n}(L)$.
To simplify the notations, we do not index $\mu$ and $\Ker$ at the kernel-based estimators.
}, respectively.
\end{definition}

\begin{corollary}\label{cor:kerappr-inter-coef}
The kernel-based estimators $s_{\vL_n,\vf_n}(L)$
and $s_{\vL_n,\hatvf_n,\epsilon_n}(L)$
in Theorems \ref{thm:kerappr-inter} and \ref{thm:kerappr-inter-noise}
can be rewritten as
\[
s_{\vL_n,\vf_n}(L)=L\mu+\vb_{\Ker,\vL_n}(L)^T\vc_n\text{ and }
s_{\vL_n,\hatvf_n,\epsilon_n}(L)=L\mu+\vb_{\Ker,\vL_n}(L)^T\hatvc_n,
\]
respectively, where $\vc_n$ and $\hatvc_n$ are the least-squared solutions of the linear systems
\[
\vA_{\Ker,\vL_n}\vc_n=\vf_n-\vL_n\mu
\text{ and }
\vA_{\Ker,\vL_n}\hatvc_n=\veta_{\Ker,\vL_n}^{\epsilon_n}\big(\hatvf_n\big)-\vL_n\mu,
\]
respectively.
\end{corollary}
\begin{proof}
Theorems \ref{thm:kerappr-inter} and \ref{thm:kerappr-inter-noise} show that the kernel-based estimators
$s_{\vL_n,\vf_n}(L)$
and $s_{\vL_n,\hatvf_n,\epsilon_n}(L)$ have the forms of
\[
s_{\vL_n,\vf_n}(L)=L\mu+\vb_{\Ker,\vL_n}(L)^T\vA_{\Ker,\vL_n}^{\dag}\left(\vf_n-\vL_n\mu\right),
\]
and
\[
s_{\vL_n,\hatvf_n,\epsilon_n}(L)
=
L\mu+\vb_{\Ker,\vL_n}(L)^T\vA_{\Ker,\vL_n}^{\dag}
\left(\veta_{\Ker,\vL_n}^{\epsilon_n}\big(\hatvf_n\big)-\vL_n\mu\right).
\]
Since $\vA_{\Ker,\vL_n}$ is a symmetric positive definite matrix,
we know that
$\vc:=\vA_{\Ker,\vL_n}^{\dag}\ve$ is a minimizer of the least-squared problem
\[
\norm{\vA_{\Ker,\vL_n}\vc-\ve}_2
=
\min_{\vz\in\RR^n}\norm{\vA_{\Ker,\vL_n}\vz-\ve}_2,
\]
for any $\ve\in\RR^n$.
Replacing $\ve$ to $\vf_n-\vL_n\mu$ or $\veta_{\Ker,\vL_n}^{\epsilon}\big(\hatvf_n\big)-\vL_n\mu$, we complete the proof.
\end{proof}

\begin{corollary}\label{cor:kerappr-inter-coef-error}
The coefficients $\vc_n$ and $\hatvc_n$ in Corollary \ref{cor:kerappr-inter-coef} have the bound
\[
\norm{\vc_n-\hatvc_n}_{2}\leq\frac{2\epsilon_n}{\lambda_{\min}\big(\vA_{\Ker,\vL_n}\big)}.
\]
\end{corollary}
\begin{proof}
Let $\ve_n:=\vf_n-\veta_{\Ker,\vL_n}^{\epsilon_n}\big(\hatvf_n\big)$.
Thus, $\vc_n-\hatvc_n=\vA_{\Ker,\vL_n}^{\dag}\ve_n$.
Since $\vf_n,\veta_{\Ker,\vL_n}^{\epsilon_n}\big(\hatvf_n\big)\in B_n\big(\hatvf_n,\epsilon_n\big)$,
we have that $\norm{\ve_n}_{2}\leq2\epsilon_n$.
Therefore, we know that
\[
\norm{\vc_n-\hatvc_n}_{2}
\leq\norm{\vA_{\Ker,\vL_n}^{\dag}}_2\norm{\ve_n}_2
\leq\frac{2\epsilon_n}{\lambda_{\min}\big(\vA_{\Ker,\vL_n}\big)}.
\]
\end{proof}

\subsection{Convergence of Kernel-based Estimators}\label{sec:ConvKerAppr}

Let $\left\{\Aset_n:n\in\NN\right\}\subseteq\Filter_{\Banach}$ be a collection of interpolation events.
Thus, we obtain the kernel-based estimators $\left\{s_{\Aset_n}(L):n\in\NN\right\}$ in Proposition \ref{pro:kerappr}.
Now we show the convergence of $s_{\Aset_n}(L)$ to the exact value $Lu$ by the kernel-based probability measure $\PP_{\Ker}^{\mu}$. Let $B_{\Banach}(0,r):=\left\{\omega\in\Banach:\norm{\omega}_{\Banach}\leq r\right\}$ be the closed ball centered at the origin $0$ with the radius $r>0$. Let $\Aset_{\infty}:=\cap_{n\in\NN}\Aset_n$.

\begin{proposition}\label{pro:kerappr-conv}
If the interpolation events $\left\{\Aset_n:n\in\NN\right\}\subseteq\Filter_{\Banach}$ satisfy that
\begin{equation}\label{eq:kerappr-conv-cond-1}
B_{\Banach}(0,r)\supseteq\Aset_1\supseteq\ldots\supseteq\Aset_n\supseteq\ldots\supseteq
\Aset_{\infty}=\left\{u\right\},
\end{equation}
then for any $L\in\Banach^{\ast}$, the kernel-based estimator $s_{\Aset_n}(L)$ converges to the exact value $Lu$ when $n\to\infty$.
\end{proposition}
\begin{proof}
Let $a:=Lu$. For $n\in\NN$, we define the probability measure
\begin{equation}\label{eq:kerappr-conv-0}
\nu_n(A):=\PP_{\Ker}^{\mu}\left(S_{L}\in A|\Aset_n\right),
\quad\text{for any open set }A\text{ in }\RR.
\end{equation}
Since $\Aset_n\in\Filter_{\Banach}$ for all $n\in\NN$, the intersection $\Aset_{\infty}\in\Filter_{\Banach}$.
The decreasing monotonicity of $\Aset_n$ in equation \eqref{eq:kerappr-conv-cond-1} shows that
\[
\lim_{n\to\infty}\PP_{\Ker}^{\mu}\left(S_{L}\in A|\Aset_n\right)
=\PP_{\Ker}^{\mu}\left(S_{L}\in A|\Aset_{\infty}\right)
\]
Thus, we have that
\[
\lim_{n\to\infty}\nu_n(A)=\delta_{a}(A),
\]
where $\delta_{a}$ is the Dirac delta measure at $a$.
The Portmanteau theorem assures that $\nu_n$ converges to $\delta_{a}$ when $n\to\infty$.
Moreover, the Skorokhod's representation theorem assures that there exist random variables $\left\{Z_n:n\in\NN\right\}$ defined on a common probability space $\left([0,1],\Filter_{[0,1]},\nu_{\ast}\right)$ such that $Z_n$ has the probability distribution $\nu_n$ and $Z_n$ converges to $a$ almost surely when $n\to\infty$,
where $\Filter_{[0,1]}$ is the Borel $\sigma$-algebra of $[0,1]$ and $\nu_{\ast}$ is the Lebesgue measure.
This shows that
\begin{equation}\label{eq:kerappr-conv-1}
\Mean\left(S_L|\Aset_n\right)
=\int_{\RR}z\nu_n(\ud z)
=\Mean(Z_n).
\end{equation}
Putting equation \eqref{eq:kerappr-2} into equation \eqref{eq:kerappr-conv-1}, we have that
\begin{equation}\label{eq:kerappr-conv-2}
s_{\Aset_n}(L)
=\Mean(Z_n).
\end{equation}

Let $M:=r\norm{L}_{\Banach^{\ast}}$. Since $\Aset_n\subseteq B_{\Banach}(0,r)$, we have that
\[
\abs{\langle\omega,L\rangle_{\Banach}}\leq\norm{L}_{\Banach^{\ast}}\norm{\omega}_{\Banach}\leq M,
\quad\text{for all }\omega\in\Aset_n.
\]
This shows that
\[
\nu_n\big([-M,M]\big)=\PP_{\Ker}^{\mu}\left(S_{L}\in[-M,M]|\Aset_n\right)=1.
\]
Thus, $\abs{Z_n}\leq M$ almost surely. By the bounded convergence theorem,
we know that $Z_n$ converges to $a$ in $\Leb_1$-based mean when $n\to\infty$.
Therefore, we have that
\begin{equation}\label{eq:kerappr-conv-3}
\lim_{n\to\infty}\Mean(Z_n)=a.
\end{equation}
Combining equations \eqref{eq:kerappr-conv-2} and \eqref{eq:kerappr-conv-3}, we conclude that
\[
\lim_{n\to\infty}s_{\Aset_n}(L)=a=Lu.
\]
\end{proof}

Given the infinite countable data $\left\{(L_n,f_n):n\in\NN\right\}\subseteq\Banach^{\ast}\times\RR$,
we have the pairs of the vectors $\left\{(\vL_n,\vf_n):n\in\NN\right\}$ such that
$\vL_n=\left(L_1,\cdots,L_n\right)^T$ and $\vf_n=\left(f_1,\cdots,f_n\right)^T$ for all $n\in\NN$.
Just like Example \ref{exa:Hermite-Birkhoff}, we can use the data $\big(\vL_n,\vf_n\big)$ and the Gaussian process $S$ to construct the interpolation event $\Aset_n$.
Thus, $u\in\Aset_n$ for all $n\in\NN$. This shows that $u\in\cap_{n\in\NN}\Aset_n=\Aset_{\infty}$.
By Theorems \ref{thm:kerappr-inter}, we can obtain the kernel-based estimator $s_{\vL_n,\vf_n}(L)$ conditioned on the data $(\vL_n,\vf_n)$.
But it is not easy to check whether $\Aset_n$ is bounded in $\Banach$. Thus, we can not use Proposition \ref{pro:kerappr-conv} to verify the convergence of $s_{\vL_n,\vf_n}(L)$ directly.
Now we will prove the convergence by the similar method of Proposition \ref{pro:kerappr-conv} without the bounded conditions.

\begin{theorem}\label{thm:kerappr-conv}
If the interpolation events $\left\{\Aset_n:n\in\NN\right\}$ based on the non-noise data $\left\{\big(\vL_n,\vf_n\big):n\in\NN\right\}$ same as Example \ref{exa:Hermite-Birkhoff} satisfy that
\begin{equation}\label{eq:kerappr-conv-cond-2}
\Aset_1\supseteq\ldots\supseteq\Aset_n\supseteq\ldots\supseteq
\Aset_{\infty}=\left\{u\right\},
\end{equation}
then for any $L\in\Banach^{\ast}$, the kernel-based estimator $s_{\vL_n,\vf_n}(L)$ converges to the exact value $Lu$ when $n\to\infty$.
\end{theorem}
\begin{proof}
Let $a:=Lu$.
By the same method of the proof in Proposition \ref{pro:kerappr-conv},
equation \eqref{eq:kerappr-conv-cond-2} assures that we can construct the random variables $Z_n$ with the probability distributions
$\nu_n$ in equation \eqref{eq:kerappr-conv-0} for all $n\in\NN$ such that
$Z_n$ converges to $a$ almost surely when $n\to\infty$.
This shows that the characteristic function $\phi_{Z_n}(t)$ of $Z_n$ converges to $\exp(ita)$
pointwisely when $n\to\infty$, that is,
\begin{equation}\label{eq:kerappr-conv-inter-1}
\lim_{n\to\infty}\phi_{Z_n}(t)
=\lim_{n\to\infty}\int_{\RR}\exp(itz)\nu_n(\ud z)
=\int_{\RR}\exp(itz)\delta_a(\ud z)
=\exp(ita),
\end{equation}
for $t\in\RR$.

Now we use the characteristic functions to prove the convergence of $s_{\vL_n,\vf_n}(L)$.
Since $S_L$ and $\vS_{\vL_n}$ have the normal distributions discussed in the proof of Theorem \ref{thm:kerappr-inter}, we know that
\[
\nu_n(A_z)=\PP_{\Ker}^{\mu}\left(S_{L}\in A_z|\Aset_n\right)
=\PP_{\Ker}^{\mu}\left(S_{L}\in A_z|\vS_{\vL_n}=\vf_n\right)
=\int_{-\infty}^{z}p_{L|\vL_n}(z|\vf_n)\ud z,
\]
for $A_z:=(-\infty,z]$ and $z\in\RR$.
This shows that $p_{L|\vL_n}\left(z|\vf_n\right)$ is the probability density function of $Z_n$.
Thus, $\Mean(Z_n)=m_{L|\vL_n}\left(\vf_n\right)=s_{\vL_n,\vf_n}(L)$ and $\Var(Z_n)=\sigma_{L|\vL_n}^2$
so that the characteristic function $\phi_{Z_n}$ can be written as
\begin{equation}\label{eq:kerappr-conv-inter-2}
\phi_{Z_n}(t)
=\exp\left(its_{\vL_n,\vf_n}(L)-\sigma_{L|\vL_n}^2t^2/2\right),
\quad\text{for }t\in\RR,
\end{equation}
where $i:=\sqrt{-1}$.
Putting equation \eqref{eq:kerappr-conv-inter-2} into equation \eqref{eq:kerappr-conv-inter-1} for $t=\sqrt{2}$,
we have that
\[
\lim_{n\to\infty}
\exp\left(i\sqrt{2}s_{\vL_n,\vf_n}(L)-\sigma_{L|\vL_n}^2\right)
=
\exp\big(i\sqrt{2}a\big).
\]
Therefore, we conclude that
\[
\lim_{n\to\infty}s_{\vL_n,\vf_n}(L)=a=Lu,
\quad
\lim_{n\to\infty}\sigma_{L|\vL_n}=0.
\]
\end{proof}

If $\Banach^{\ast}$ is a separable normed space, then
there exists countable $\left\{L_n:n\in\NN\right\}\subseteq\Banach^{\ast}$ such that
$\Span\left\{L_n:n\in\NN\right\}$ is dense in $\Banach^{\ast}$.
By the density of the data, we can still verify the convergence.

\begin{corollary}\label{cor:kerappr-conv}
If $\Span\left\{L_n:n\in\NN\right\}$ is dense in $\Banach^{\ast}$, then for any $L\in\Banach^{\ast}$, the kernel-based estimator $s_{\vL_n,\vf_n}(L)$
converges to the exact value $Lu$ when $n\to\infty$.
\end{corollary}
\begin{proof}
We primarily prove that equation \eqref{eq:kerappr-conv-cond-2} is true.
Since $\vrho_n(\vf_{n+1})=\vf_n$ for $n\in\NN$, we have that $\Aset_{n}\supseteq\Aset_{n+1}$.
Next, we will prove that $\Aset_{\infty}=\{u\}$. Since $\vf_n=\vL_nu$, we have that $u\in\Aset_n$.
Thus $u\in\Aset_{\infty}$.
We take any $P\in\Banach^{\ast}$ and any $\omega\in\Aset_{\infty}$.
If we verify that $P\omega=Pu$, then we know that $\omega=u$.
Let $\varepsilon>0$. Since $\Span\left\{L_n:n\in\NN\right\}$ is dense in $\Banach^{\ast}$, there exists $P_{\varepsilon}\in\Span\left\{L_n:n\in\NN\right\}$ such that $\norm{P_{\varepsilon}-P}_{\Banach^{\ast}}\leq\varepsilon$.
Since $\langle\omega,P_{\varepsilon}\rangle_{\Banach}=\langle u,P_{\varepsilon}\rangle_{\Banach}$, we know that
\[
\abs{\langle\omega-u,P\rangle_{\Banach}}\leq\abs{\langle\omega,P-P_{\varepsilon}\rangle_{\Banach}}
+\abs{\langle u,P_{\varepsilon}-P\rangle_{\Banach}}
\leq\norm{P_{\varepsilon}-P}_{\Banach^{\ast}}\left(\norm{\omega}_{\Banach}+\norm{u}_{\Banach}\right)
=\Order(\varepsilon).
\]
Taking $\varepsilon\to0$, we have that $\abs{\langle\omega-u,P\rangle_{\Banach}}=0$.
Therefore, $\Aset_{\infty}=\{u\}$.
According to Theorem \ref{thm:kerappr-conv}, we conclude the convergence of $s_{\vL_n,\vf_n}(L)$ to $Lu$.
\end{proof}

By Theorem \ref{thm:kerappr-inter-noise}, we look at the kernel-based estimators $\big\{s_{\vL_n,\hatvf_n,\epsilon_n}(L):n\in\NN\big\}$ conditioned on
the noisy data
$\big\{\big(\vL_n,\hatvf_n\big):n\in\NN\big\}$ and the noise margins $\left\{\epsilon_n:n\in\NN\right\}$.
Thus, we can use the noisy data $\big(\vL_n,\hatvf_n\big)$, the noise margin $\epsilon_n$, and the Gaussian process $S$ to construct the interpolation event $\Aset_n^{\epsilon_n}$ same as Example \ref{exa:noise-Hermite-Birkhoff}.
By the discussions in Corollary \ref{cor:kerappr-inter-coef-error}, the representation of $s_{\vL_n,\hatvf_n,\epsilon_n}(L)$ indicates that the noise margin $\epsilon_n$ and the eigenvalues of $\vA_{\Ker,\vL_n}$ are required to be correlated for the proof of convergence.
Let $\Aset_{\infty}^{0}:=\cap_{n\in\NN}\Aset_n^{\epsilon_n}$.

\begin{theorem}\label{thm:kerappr-conv-noise}
If the interpolation events $\left\{\Aset_n^{\epsilon_n}:n\in\NN\right\}$ based on the noisy data $\big\{\big(\vL_n,\hatvf_n\big):n\in\NN\big\}$ and the noise margins $\left\{\epsilon_n:n\in\NN\right\}$ same as Example \ref{exa:noise-Hermite-Birkhoff} satisfy that
\begin{equation}\label{eq:kerappr-conv-cond-3}
\Aset_1^{\epsilon_1}\supseteq\ldots\supseteq\Aset_n^{\epsilon_n}\supseteq\ldots\supseteq
\Aset_{\infty}^{0}=\left\{u\right\}.
\end{equation}
and the noise margins $\left\{\epsilon_n:n\in\NN\right\}$ satisfy that
\begin{equation}\label{eq:kerappr-conv-cond-4}
\lim_{n\to\infty}\frac{\epsilon_n^2}{\lambda_{\min}\big(\vA_{\Ker,\vL_n}\big)}=0,
\end{equation}
then for any $L\in\Banach^{\ast}$, the kernel-based estimator $s_{\vL_n,\hatvf_n,\epsilon_n}(L)$
converges to the exact value $Lu$ when $n\to\infty$.
\end{theorem}
\begin{proof}
We will prove the convergence of $s_{\vL_n,\hatvf_n,\epsilon_n}(L)$ by the same method of the proof of Theorem~\ref{thm:kerappr-conv}. Let $a:=Lu$.
We replace $\Aset_n$ to $\Aset_n^{\epsilon_n}$ in Theorem \ref{thm:kerappr-conv}. Equation \eqref{eq:kerappr-conv-cond-3} assures that we can construct the random variables $Z_n$ with the probability distributions
$\nu_n$ for all $n\in\NN$ such that
the characteristic function $\phi_{Z_n}(t)$ of $Z_n$ converges to $\exp(ita)$
pointwisely when $n\to\infty$. Here, the probability measure $\nu_n$ is redefined by
\[
\nu_n(A):=\PP_{\Ker}^{\mu}\left(S_{L}\in A\big|\vS_{\vL_n}\in B_n\big(\hatvf_n,\epsilon_n\big)\right),
\quad
\text{for any open set }A\text{ in }\RR.
\]
Thus, the probability density function $p_{Z_n}$ of $Z_n$ can be written as
\[
p_{Z_n}(z):=
\frac{\int_{B_n(\hatvf_n,\epsilon_n)}p_{L,\vL_n}(z,\vv)\ud\vv}
{\int_{B_n(\hatvf_n,\epsilon_n)}p_{\vL_n}(\vv)\ud\vv},
\quad\text{for }z\in\RR.
\]

Next, we compute the characteristic function
\begin{equation}\label{eq:kerappr-conv-noise-1}
\phi_{Z_n}(t)=\Mean(\exp(itZ_n))
=\int_{\RR}\exp(itz)p_{Z_n}(z)\ud z,
\quad
\text{for }t\in\RR.
\end{equation}
According to the mean value theorem,
there exists $\vzeta_n\in B_n\big(\hatvf_n,\epsilon_n\big)$ such that
\begin{equation}\label{eq:kerappr-conv-noise-2}
\begin{split}
&\int_{\RR}\int_{B_n(\hatvf_n,\epsilon_n)}\exp(itz)p_{L,\vL_n}(z,\vv)\ud\vv\ud z\\
=&\int_{B_n(\hatvf_n,\epsilon_n)}
p_{\vL_n}(\vv)\int_{\RR}\exp(itz)p_{L|\vL_n}(z|\vv)\ud z\ud\vv\\
=&\int_{B_n(\hatvf_n,\epsilon_n)}
\exp\left(itm_{L|\vL_n}\left(\vv\right)-\sigma_{L|\vL_n}^2t^2/2\right)p_{\vL_n}(\vv)\ud\vv\\
=&\exp\left(itm_{L|\vL_n}\big(\vzeta_n\big)-\sigma_{L|\vL_n}^2t^2/2\right)
\int_{B_n(\hatvf_n,\epsilon_n)}
p_{\vL_n}(\vv)\ud\vv
\end{split}.
\end{equation}
Putting equation \eqref{eq:kerappr-conv-noise-2} into equation \eqref{eq:kerappr-conv-noise-1},
we have that
\[
\phi_{Z_n}(t)=\exp\left(itm_{L|\vL_n}\big(\vzeta_n\big)-\sigma_{L|\vL_n}^2t^2/2\right).
\]
This shows that
\begin{equation}\label{eq:kerappr-conv-noise-3}
\lim_{n\to\infty}m_{L|\vL_n}\big(\vzeta_n\big)=a.
\end{equation}
Let $\ve_n:=\veta_{\Ker,\vL_n}^{\epsilon_n}\big(\hatvf_n\big)-\vzeta_n$.
The symmetric positive definite matrix $\vA_{\Ker,\vL_n}^{\dag}$ guarantees that
\begin{equation}\label{eq:kerappr-conv-noise-5}
\begin{split}
&\abs{s_{\vL_n,\hatvf_n,\epsilon_n}(L)-m_{L|\vL_n}\big(\vzeta_n\big)}^2
=
\abs{\vb_{\Ker,\vL_n}(L)^T\vA_{\Ker,\vL_n}^{\dag}\ve_n}^2\\
\leq&\left(\vb_{\Ker,\vL_n}(L)^T\vA_{\Ker,\vL_n}^{\dag}\vb_{\Ker,\vL_n}(L)\right)
\left(\ve_n^T\vA_{\Ker,\vL_n}^{\dag}\ve_n\right).
\end{split}
\end{equation}
Since $\sigma_{L|\vL_n}^2\geq0$, we have that
\begin{equation}\label{eq:kerappr-conv-noise-6}
\vb_{\Ker,\vL_n}(L)^T\vA_{\Ker,\vL_n}^{\dag}\vb_{\Ker,\vL_n}(L)\leq\Ker(L,L).
\end{equation}
Moreover, since $\veta_{\Ker,\vL_n}^{\epsilon_n}\big(\hatvf_n\big),\vzeta_n\in B_n\big(\hatvf_n,\epsilon_n\big)$,
we know that
\[
\norm{\ve_n}_{2}\leq2\epsilon_n.
\]
This shows that
\begin{equation}\label{eq:kerappr-conv-noise-7}
\ve_n^T\vA_{\Ker,\vL_n}^{\dag}\ve_n
\leq
\frac{\norm{\ve_n}_2^2}{\lambda_{\min}\big(\vA_{\Ker,\vL_n}\big)}
\leq
\frac{4\epsilon_n^2}{\lambda_{\min}\big(\vA_{\Ker,\vL_n}\big)}.
\end{equation}
Putting equations \eqref{eq:kerappr-conv-noise-6} and \eqref{eq:kerappr-conv-noise-7}
into equation \eqref{eq:kerappr-conv-noise-5}, we have that
\begin{equation}\label{eq:kerappr-conv-noise-8}
\abs{s_{\vL_n,\hatvf_n,\epsilon_n}(L)-m_{L|\vL_n}\big(\vzeta_n\big)}^2
\leq\frac{4\Ker(L,L)\epsilon_n^2}{\lambda_{\min}\big(\vA_{\Ker,\vL_n}\big)}.
\end{equation}
Taking the limits of the both sides of equation \eqref{eq:kerappr-conv-noise-8}, the limit condition in equations \eqref{eq:kerappr-conv-cond-4} shows that
\begin{equation}\label{eq:kerappr-conv-noise-9}
\lim_{n\to\infty}s_{\vL_n,\hatvf_n,\epsilon_n}(L)=\lim_{n\to\infty}m_{L|\vL_n}\big(\vzeta_n\big).
\end{equation}
Therefore, combining equations \eqref{eq:kerappr-conv-noise-3} and \eqref{eq:kerappr-conv-noise-9},
we conclude that
\[
\lim_{n\to\infty}s_{\vL_n,\hatvf_n,\epsilon_n}(L)=a=Lu.
\]
\end{proof}

\begin{corollary}\label{cor:kerappr-conv-noise}
Suppose that
\begin{equation}\label{eq:noise-data-cond}
\hatvf_{n}+B_n\big(0,\epsilon_n\big)\supseteq
\vrho_{n}\big(\hatvf_{n+1}\big)+B_n\big(0,\epsilon_{n+1}\big),
\quad\text{for all }n\in\NN.
\end{equation}
If $\Span\left\{L_n:n\in\NN\right\}$ is dense in $\Banach^{\ast}$
and the noise margins $\left\{\epsilon_n:n\in\NN\right\}$ satisfy that
\[
\lim_{n\to\infty}\epsilon_n=0
\text{ and }
\lim_{n\to\infty}\frac{\epsilon_n^2}{\lambda_{\min}\big(\vA_{\Ker,\vL_n}\big)}=0,
\]
then for any $L\in\Banach^{\ast}$, the kernel-based estimator $s_{\vL_n,\hatvf_n,\epsilon_n}(L)$
converges to the exact value $Lu$ when $n\to\infty$.
\end{corollary}
\begin{proof}
If we prove that equation \eqref{eq:kerappr-conv-cond-3} is true, then the convergence of $s_{\vL_n,\hatvf_n,\epsilon_n}(L)$ can be guaranteed by Theorem \ref{thm:kerappr-conv-noise}.
Equation \eqref{eq:noise-data-cond} shows that
$\Aset_{n}^{\epsilon_{n}}\supseteq\Aset_{n+1}^{\epsilon_{n+1}}$ for all $n\in\NN$.
Since $\vL_nu\in B_n\big(\hatvf_n,\epsilon_n\big)$, we have that $u\in\Aset_{n}^{\epsilon_{n}}$ for all $n\in\NN$.
Thus, $u\in\Aset_{\infty}^{0}$.
We take any $\omega\in\Aset_{\infty}^{0}$.
For $P\in\Span\left\{L_n:n\in\NN\right\}$, equation \eqref{eq:noise-data-cond} assures that $\abs{\langle\omega-u,P\rangle_{\Banach}}=\Order(\epsilon_n)$ for all $n\in\NN$.
Since $\epsilon_n\to0$ when $n\to\infty$, we conclude that $\abs{\langle\omega-u,P\rangle_{\Banach}}=0$.
This shows that $P\omega=Pu$.
By the same method of Corollary~\ref{cor:kerappr-conv}, we can verify that $\omega=u$.
This shows that $\Aset_{\infty}^{0}=\{u\}$.
Therefore, equation \eqref{eq:kerappr-conv-cond-3} is true.
\end{proof}

\subsection{Error Bounds of Kernel-based Estimators}\label{sec:ErrorKerAppr}

Finally, we investigate the error bounds of $Lu-s_{\vL_n,\vf_n}(L)$ and
$Lu-s_{\vL_n,\hatvf_n,\epsilon_n}(L)$ for a special class of the solutions $u$, respectively.
By Proposition \ref{pro:SL-L}, we have that
\[
\int_{\Banach}\norm{\omega S_L(\omega)}_{\Banach}\PP_{\Ker}^{\mu}(\ud\omega)
\leq C\norm{L}_{\Banach^{\ast}}<\infty,
\quad\text{for }L\in\Banach^{\ast}.
\]
This shows that the map
\[
\Gamma_{\Ker}^{\mu}(S_L):=\int_{\Banach}\omega S_L(\omega)\PP_{\Ker}^{\mu}(\ud\omega),
\]
is a bounded linear operator from $\Span\left\{S_L:L\in\Banach^{\ast}\right\}$ into $\Banach$.
By the competition $\Hilbert_S$ of $\Span\left\{S_L:L\in\Banach^{\ast}\right\}$, the map $\Gamma_{\Ker}^{\mu}$ can be extended to the Hilbert space $\Hilbert_S$, that is,
\begin{equation}\label{eq:Gamma-map}
\Gamma_{\Ker}^{\mu}(V):=\int_{\Banach}\omega V(\omega)\PP_{\Ker}^{\mu}(\ud\omega),\quad
\text{for }V\in\Hilbert_S.
\end{equation}
This shows that the range of $\Gamma_{\Ker}^{\mu}$ is a subspace of $\Banach$.
Moreover, the map $\Gamma_{\Ker}^{\mu}$ is the adjoint operator of the imbedding map from $\Banach^{\ast}$ into $\Hilbert_S$ in Proposition \ref{pro:SL-L-imbedding}.
By the structure theorem of Gaussian measures in \cite[Lemma 2.1]{Kuelbs1976} and \cite[Lemma 8.2.3]{Stroock1993},
if there exists $U\in\Hilbert_S$ such that $u=\Gamma_{\Ker}^{\mu}(U)$, then
\begin{equation}\label{eq:dual-mean-U-u}
\langle u,L\rangle_{\Banach}=\int_{\Banach}\langle\omega,L\rangle_{\Banach}U(\omega)\PP_{\Ker}^{\mu}(\ud\omega)
=
\int_{\Banach}S_L(\omega)U(\omega)\PP_{\Ker}^{\mu}(\ud\omega)
=\Mean(S_LU).
\end{equation}
Based on the map $\Gamma_{\Ker}^{\mu}$, we verify that the errors of $s_{\vL_n,\vf_n}(L)$ and $s_{\vL_n,\hatvf_n,\epsilon_n}(L)$ can be
bounded by $\sigma_{L|\vL_n}$ in equation \eqref{eq:var}, respectively.
We define a vector function
\[
\vbeta_{\Ker,\vL_n}(L):=\vA_{\Ker,\vL_n}^{\dag}\vb_{\Ker,\vL_n}(L),
\quad\text{for }L\in\Banach^{\ast}.
\]
Thus, $\vbeta_{\Ker,\vL_n}$ is a map from $\Banach^{\ast}$ into $\RR^n$.
Since $\vA_{\Ker,\vL_n}$ is a symmetric positive definite matrix,
the vector $\vbeta_{\Ker,\vL_n}(L)$ is the least-squared solution of the linear system $\vA_{\Ker,\vL_n}\vz=\vb_{\Ker,\vL_n}(L)$.

\begin{theorem}\label{thm:kerappr-error}
If the exact solution $u\in\range\big(\Gamma_{\Ker}^{\mu}\big)$, then for any $L\in\Banach^{\ast}$, the kernel-based estimator $s_{\vL_n,\vf_n}(L)$ has the error bound
\begin{equation}\label{eq:kerappr-error-bound-1}
\abs{Lu-s_{\vL_n,\vf_n}(L)}\leq C_1\sigma_{L|\vL_n}+C_2\abs{L\mu-\vbeta_{\Ker,\vL_n}(L)^T\vL_n\mu},
\end{equation}
where the constants $C_1\geq0$ and $C_2\geq0$ are independent of $L$, $\vL_n$, and $\vf_n$.
\end{theorem}
\begin{proof}
Since $u\in\range\big(\Gamma_{\Ker}^{\mu}\big)$, there exists $U\in\Hilbert_S$ such that $u=\Gamma_{\Ker}^{\mu}(U)$.
Let $Y_L:=S_L-L\mu$ and $\vY_{\vL_n}:=\vS_{\vL_n}-\vL_n\mu$.
Thus, $Y_L\sim\Normal(0,\Ker(L,L))$,
$\vY_{\vL_n}\sim\Normal\big(0,\vA_{\Ker,\vL_n}\big)$, and $\Mean\left(Y_L\vY_{\vL_n}\right)=\vb_{\Ker,\vL_n}(L)$.
Let $\gamma:=\Mean(U)$.
Equation \eqref{eq:dual-mean-U-u} shows that
\begin{equation}\label{eq:kerappr-error-1}
Lu-\gamma L\mu=\Mean(S_LU)-\Mean(L\mu U)=\Mean(Y_{L}U),
\end{equation}
and
\begin{equation}\label{eq:kerappr-error-2}
\vf_n-\gamma\vL_n\mu=
\vL_nu-\gamma\vL_n\mu=
\Mean(\vS_{\vL_n}U)-\Mean(\vL_n\mu U)=\Mean(\vY_{\vL_n}U).
\end{equation}
Since $s_{\vL_n,\vf_n}(L)=L\mu+\vbeta_{\Ker,\vL_n}(L)^T(\vf_n-\vL_n\mu)$,
we have that
\begin{equation}\label{eq:kerappr-error-3}
Lu-s_{\vL_n,\vf_n}(L)
=(Lu-\gamma L\mu)-\vbeta_{\Ker,\vL_n}(L)^T(\vf_n-\gamma\vL_n\mu)
+(\gamma-1)\left(L\mu-\vbeta_{\Ker,\vL_n}(L)^T\vL_n\mu\right).
\end{equation}
Putting equations \eqref{eq:kerappr-error-1} and \eqref{eq:kerappr-error-2} into equation \eqref{eq:kerappr-error-3},
we know that
\begin{equation}\label{eq:kerappr-error-4}
\begin{split}
&\abs{Lu-s_{\vL_n,\vf_n}(L)}
=
\abs{\Mean\left(\left(Y_L-\vbeta_{\Ker,\vL_n}(L)^T\vY_{\vL_n}\right)U\right)
+(\gamma-1)\left(L\mu-\vbeta_{\Ker,\vL_n}(L)^T\vL_n\mu\right)}\\
\leq&\sqrt{\Mean\abs{U}^2}\sqrt{\Mean\abs{Y_L-\vbeta_{\Ker,\vL_n}(L)^T\vY_{\vL_n}}^2}
+\abs{\gamma-1}\abs{L\mu-\vbeta_{\Ker,\vL_n}(L)^T\vL_n\mu}.
\end{split}
\end{equation}
Moreover, we compute the mean square
\begin{equation}\label{eq:kerappr-error-5}
\begin{split}
&\Mean\abs{Y_L-\vbeta_{\Ker,\vL_n}(L)^T\vY_{\vL_n}}^2\\
=&
\Mean\abs{Y_L}^2-2\vbeta_{\Ker,\vL_n}(L)^T\Mean(Y_L\vY_{\vL_n})
+\vbeta_{\Ker,\vL_n}(L)^T\Mean(\vY_{\vL_n}\vY_{\vL_n}^T)\vbeta_{\Ker,\vL_n}(L)\\
=&
\Ker(L,L)-2\vbeta_{\Ker,\vL_n}(L)^T\vk_{\Ker,\vL_n}(L)+\vbeta_{\Ker,\vL_n}(L)^T\vA_{\Ker,\vL_n}\vbeta_{\Ker,\vL_n}(L)
=\sigma_{L|\vL_n}^2.
\end{split}
\end{equation}
Let $C_1:=\sqrt{\Mean\abs{U}^2}$ and $C_2:=\abs{\gamma-1}$. Combining equations \eqref{eq:kerappr-error-4} and \eqref{eq:kerappr-error-5},
we complete the proof.
\end{proof}

\begin{corollary}\label{cor:kerappr-error}
Suppose that the mean $\mu=0$.
If the exact solution $u\in\range\big(\Gamma_{\Ker}^{\mu}\big)$, then for any $L\in\Banach^{\ast}$, the kernel-based estimator $s_{\vL_n,\vf_n}(L)$ has the error bound
\[
\abs{Lu-s_{\vL_n,\vf_n}(L)}\leq C_1\sigma_{L|\vL_n},
\]
where the constant $C_1\geq0$ is independent of $L$, $\vL_n$, and $\vf_n$.
\end{corollary}
\begin{proof}
The proof is completed by Theorem \ref{thm:kerappr-error} immediately.
\end{proof}

If the covariance kernel $\Ker$ is defined by the positive definite kernel $K$, then
for $L:=\delta_{\vx}$ and $\vL_n:=\left(\delta_{\vx_1},\cdots,\delta_{\vx_n}\right)^T$, the standard deviation $\sigma_{L|\vL_n}$ is equal to the classical power function
\[
q_{K,X}(\vx):=\sqrt{K(\vx,\vx)-\vb_{K,X}(\vx)^T\vA_{K,X}^{-1}\vb_{K,X}(\vx)},
\]
where $\vb_{K,X}(\vx):=\left(K(\vx,\vx_1),\cdots,K(\vx,\vx_n)\right)^T$. This shows that $\sigma_{L|\vL_n}$ can be viewed as a generalization of the classical power functions in \cite{Fasshauer2007,Wendland2005} and kriging functions in \cite{WuSchaback1993}.
Therefore, for the special $\sigma_{L|\vL_n}$, we can compute the convergent rates by using fill distances same as meshfree methods in \cite[Chapters 14 and 15]{Fasshauer2007} and \cite[Chapters 11 and 16]{Wendland2005}.

\begin{theorem}\label{thm:kerappr-error-noise}
If the exact solution $u\in\range\big(\Gamma_{\Ker}^{\mu}\big)$, then for any $L\in\Banach^{\ast}$, the kernel-based estimator $s_{\vL,\hatvf_n,\epsilon_n}(L)$ has the error bound
\[
\abs{Lu-s_{\vL,\hatvf_n,\epsilon_n}(L)}\leq C_1\sigma_{L|\vL_n}
+C_2\abs{L\mu-\vbeta_{\Ker,\vL_n}(L)^T\vL_n\mu}
+M_{L,\vL_n}\epsilon_n,
\]
where the constants $C_1\geq0$ and $C_2\geq0$ are independent of $L$, $\vL_n$, $\hatvf_n$, and $\epsilon_n$, and
the constant $M_{L,\vL_n}\geq0$ is independent of $\hatvf_n$ and $\epsilon_n$.
\end{theorem}
\begin{proof}
By the same method of the proof of Theorem~\ref{thm:kerappr-conv-noise} such as equation \eqref{eq:kerappr-conv-noise-8},
we have that
\begin{equation}\label{eq:kerappr-error-noise-1}
\abs{s_{\vL_n,\vf_n}(L)-s_{\vL_n,\hatvf_n,\epsilon_n}(L)}^2
=
\abs{\vb_{\Ker,\vL_n}(L)^T\vA_{\Ker,\vL_n}^{\dag}\left(\vf_n-\veta_{\Ker,\vL_n}^{\epsilon_n}\big(\hatvf_n\big)\right)}^2
\leq
\frac{4\Ker(L,L)\epsilon_n^2}{\lambda_{\min}\big(\vA_{\Ker,\vL_n}\big)}.
\end{equation}
Let
\[
M_{L,\vL_n}:=
2\sqrt{\frac{\Ker(L,L)}{\lambda_{\min}\big(\vA_{\Ker,\vL_n}\big)}}.
\]
Putting equations \eqref{eq:kerappr-error-bound-1} and \eqref{eq:kerappr-error-noise-1} into
\[
\abs{Lu-s_{\vL_n,\hatvf_n,\epsilon_n}(L)}
\leq
\abs{Lu-s_{\vL_n,\vf_n}(L)}+\abs{s_{\vL_n,\vf_n}(L)-s_{\vL_n,\hatvf_n,\epsilon_n}(L)},
\]
we complete the proof.
\end{proof}

Here, we only consider the situation that the noise $\vxi_n$ satisfies $\norm{\vxi_n}_2\leq\epsilon_n$.
Actually, we can construct and analyze another kernel-based estimator for many kinds of noisy data by the same methods in this section.
For example, the noise $\vxi_n$ is reconsidered to satisfy that $\norm{\vxi_n}_{\infty}\leq\epsilon_n$.
Thus, we replace the closed ball $\hatvf_n+B_n\big(0,\epsilon_n\big)$ to the closed cube $\hatvf_n+[-\epsilon_n,\epsilon_n]^n$ to reconstruct the kernel-based estimator $s_{\vL_n,\hatvf_n,\epsilon_n}(L)$ in equation \eqref{eq:kerappr-4}.
To be more precise, $\veta_{\Ker,\vL_n}^{\epsilon_n}\big(\hatvf_n\big)$
is recomputed from the average over $\hatvf_n+[-\epsilon_n,\epsilon_n]^n$ measured by $\PP_{\Ker}^{\mu}$.

There are many applications for generalized interpolations by kernel-based probability measures. Specially, we will use the kernel-based estimators to approximate the solutions of elliptic partial differential equations in the next section.

\section{Constructions of Numerical Solutions of Elliptic Partial Differential Equations by Kernel-based Probability Measures}\label{sec:MeshfreeApproxPDE}

In this section, we solve an elliptic problem by the kernel-based probability measures.
Suppose that the domain $\Domain\subseteq\Rd$ is regular and compact.
We want to solve an elliptic equation with a Dirichlet boundary condition
\begin{equation}\label{eq:pde}
\begin{cases}
\Delta u = f,&\text{in }\Domain,\\
~~u=g,&\text{on }\partial\Domain,
\end{cases}
\end{equation}
where $f\in\Cont(\Domain)$ and $g\in\Cont(\partial\Domain)$ such that $u\in\Hilbert^m(\Domain)$ for $m>d/2$.
By the maximum principle, elliptic equation \eqref{eq:pde} exits the unique solution $u$.
Sobolev imbedding theorem assures that $\Hilbert^m(\Domain)$ is imbedding in $\Cont(\Domain)$. Thus, we know that $\delta_{\vx}\in\Hilbert^m(\Domain)^{\ast}$ for $\vx\in\Domain$.

Let the mollifiers
\[
\varphi_r(\vx):=r^{-d}\varphi\big(r^{-1}\vx\big),\quad\text{for }\vx\in\Rd\text{ and }r>0,
\]
and
\[
\varphi(\vx):=
\begin{cases}
\kappa\exp\left(-\left(1-\norm{\vx}_2^2\right)^{-1}\right),&\text{if }\norm{\vx}_2<1,\\
0,&\text{otherwise},
\end{cases}
\]
where the constant $\kappa$ is chosen such that $\int_{\Rd}\varphi(\vx)\ud\vx=1$.
Clearly $\varphi\in\Cont_0^{\infty}(\Rd)$.
For the data point $(\vx,r)\in\Domain^{\circ}\times\RR_+$ such that $B_d(\vx,r)\subseteq\Domain$, we define a linear functional
\begin{equation}\label{eq:pde-L-varphi}
L_{(\vx,r)}(\omega):=\int_{\Domain}\nabla\omega(\vy)^T\nabla\varphi_{r}(\vy-\vx)\ud\vy,
\quad\text{for }\omega\in\Hilbert^1(\Domain),
\end{equation}
where $\nabla$ is a gradient.
Thus, $L_{(\vx,r)}\in\Hilbert^m(\Domain)^{\ast}$. We also define a scalar
\begin{equation}\label{eq:pde-f-varphi}
h(\vx,r):=\int_{\Domain}f(\vy)\varphi_{r}(\vy-\vx)\ud\vy.
\end{equation}

Let $(\vx_1,r_1),\ldots,(\vx_{n_1},r_{n_1})\in\Domain^{\circ}\times\RR_+$ such that
$B_d(\vx_1,r_1),\ldots,B_d(\vx_{n_1},r_{n_1})\subseteq\Domain$.
Thus, we have the test functions $\varphi_{r_k}(\cdot-\vx_k)\in\Cont_0^{\infty}(\Domain)$ for $k=1,\ldots,n_1$.
Let $\vz_{1},\ldots,\vz_{n_2}\in\partial\Domain$. Based on the data points, we obtain the data
\[
\vL_n:=\big(L_{(\vx_1,r_1)},\cdots,L_{(\vx_{n_1},r_{n_1})},\delta_{\vz_1},\cdots,\delta_{\vz_{n_2}}\big)^T,
\]
and
\[
\vf_n:=\big(h(\vx_1,r_1),\cdots,h(\vx_{n_1},r_{n_1}),g(\vz_1),\cdots,g(\vz_{n_2})\big)^T,
\]
where $n:=n_1+n_2$.
Elliptic equation \eqref{eq:pde} shows that $\vL_nu=\vf_n$.

Let $K$ be a symmetric strictly positive definite kernel on $\Domain$.
Suppose that $\mu\in\Hilbert^m(\Domain)$ and $K\in\Cont^{2m,1}(\Domain\times\Domain)$.
Just like Example \ref{exa:PDK-Sobolev},
there exists the kernel-based probability measure $\PP_{\Ker}^{\mu}$ on $\Hilbert^m(\Domain)$ with the mean $\mu$ and the covariance kernel
$\Ker(L_1,L_2):=L_{1,\vx}L_{2,\vy}K(\vx,\vy)$ for $L_1,L_2\in\Hilbert^m(\Domain)^{\ast}$.
Thus, we have that
\[
\Ker\big(\delta_{\vx},\delta_{\vz}\big)=K(\vx,\vz),
\quad
\Ker\left(\delta_{\vx},L_{(\vz,r)}\right)
=\int_{\Domain}\nabla_{\vy}K(\vx,\vy)^T\nabla_{\vy}\varphi_{r}(\vy-\vz)\ud\vx,
\]
and
\[
\Ker\big(L_{(\vx,r)},L_{(\vz,\gamma)}\big)=\int_{\Domain}\int_{\Domain}
\nabla_{\vy}\varphi_{r}(\vy-\vx)^T
\nabla_{\vy}\nabla_{\vv}^TK(\vy,\vv)
\nabla_{\vv}\varphi_{\gamma}(\vv-\vz)\ud\vy\ud\vv,
\]
for $\vx,\vz\in\Domain$ and $r,\gamma>0$.
We define a vector function
\[
\vvarphi_{n_1}(\vx):=\left(\varphi_{r_1}(\vx-\vx_1),\cdots,\varphi_{r_{n_1}}(\vx-\vx_{n_1})\right)^T,
\quad\text{for }\vx\in\Domain.
\]
Thus, the vectors $\vb_{\Ker,\vL_n}\big(\delta_{\vx}\big)$ and $\vL_n\mu$ have the forms of
\[
\vb_{K,\vL_n}(\vx):=\vb_{\Ker,\vL_n}\big(\delta_{\vx}\big)=
\begin{pmatrix}
\vb_{n_1}(\vx)\\
\vb_{n_2}(\vx)
\end{pmatrix}
\text{ and }
\vzeta_n:=\vL_n\mu=
\begin{pmatrix}
\vzeta_{n_1}\\\vzeta_{n_2}
\end{pmatrix}
\]
where
\[
\vb_{n_1}(\vx):=\int_{\Domain}\nabla_{\vy}K(\vx,\vy)^T\nabla_{\vy}\vvarphi_{n_1}(\vy)\ud\vy,
\quad
\vb_{n_2}(\vx):=\left(K(\vx,\vz_1),\cdots,K(\vx,\vz_{n_2})\right)^T,
\]
and
\[
\vzeta_{n_1}:=\int_{\Domain}\nabla\vvarphi_{n_1}(\vx)^T\nabla\mu(\vx)\ud\vx,
\quad
\vzeta_{n_2}:=\left(\mu(\vz_1),\cdots,\mu(\vz_{n_2})\right)^T,
\]
respectively.
Moreover, the covariance matrix $\vA_{\Ker,\vL_n}$ has the form of
\[
\vA_{\Ker,\vL_n}=
\begin{pmatrix}
\vA_{11}&\vA_{12}\\
\vA_{21}&\vA_{22}
\end{pmatrix},
\]
where
\[
\vA_{11}:=\int_{\Domain}\int_{\Domain}
\nabla_{\vx}\vvarphi_{n_1}(\vx)^T
\nabla_{\vx}\nabla_{\vy}^TK(\vx,\vy)
\nabla_{\vy}\vvarphi_{n_1}(\vy)\ud\vx\ud\vy,
\]
and
\[
\vA_{12}=\vA_{21}^T
:=\left(\vb_{n_1}(\vz_1),\cdots,\vb_{n_1}(\vz_{n_2})\right),
\quad
\vA_{22}=\left(\vb_{n_2}(\vz_1),\cdots,\vb_{n_2}(\vz_{n_2})\right).
\]
Here, $\nabla\vvarphi_n$ is the Jacobian matrix of $\vvarphi_n$ and $\nabla_{\vx}\nabla_{\vy}^T$ is an operator matrix.
According to Theorem~\ref{thm:kerappr-inter} and Corollary~\ref{cor:kerappr-inter-coef}, we have the kernel-based estimator
\begin{equation}\label{eq:pde-kerappr}
\hat{u}_{\vL_n,\vf_n}(\vx):=s_{\vL_n,\vf_n}(\delta_{\vx})
=\mu(\vx)+\vb_{K,\vL_n}(\vx)^T\vc_n,
\end{equation}
where $\vc_n$ is the least-squared solution of the linear system
\[
\vA_{\Ker,\vL_n}\vc_n=\vf_n-\vzeta_n.
\]
Clearly $\hat{u}_{\vL_n,\vf_n}$ can be viewed as a function on $\Domain$.
Through Equation \eqref{eq:pde-kerappr}, we have that $\hat{u}_{\vL_n,\vf_n}\in\Hilbert^m(\Domain)$. We call $\hat{u}_{\vL_n,\vf_n}$ a \emph{kernel-based approximate function} conditioned on the data $(\vL_n,\vf_n)$.

Finally, we show the convergence of $\hat{u}_{\vL_n,\vf_n}$.
Suppose that $X:=\left\{\vx_{n_1}:n_1\in\NN\right\}$ is a dense subset of $\Domain^{\circ}$
and $Z:=\left\{\vz_{n_2}:n_2\in\NN\right\}$ is a dense subset of $\partial\Domain$.
Let $\Theta:=\left\{r_{n_3}:n_3\in\NN\right\}$ be a decrease subsequence such that $r_{n_3}\to0$ when $n_3\to\infty$.
Using the data points $X\times\Theta$ and $Z$, we construct the functionals
\begin{equation}\label{eq:pde-Lambda}
\Lambda:=\big\{L_{(\vx_{n_1},r_{n_3})},~\delta_{\vz_{n_2}}:\vx_{n_1}\in X,~r_{n_3}\in\Theta,~\vz_{n_2}\in Z\big\},
\end{equation}
as shown in equation \eqref{eq:pde-L-varphi}.
Moreover, we also have the related data values
\begin{equation}\label{eq:pde-Xi}
\Xi:=\left\{h(\vx_{n_1},r_{n_3}),~g(\vz_{n_2}):\vx_{n_1}\in X,~r_{n_3}\in\Theta,~\vz_{n_2}\in Z\right\},
\end{equation}
as shown in equation \eqref{eq:pde-f-varphi}. Since $\Lambda$ and $\Xi$ are the countable sets,
the set $\Lambda\times\Xi$ can be reordered as $\left\{(L_n,f_n):n\in\NN\right\}$.
Thus, we have the data $\vL_n=\left(L_1,\cdots,L_n\right)^T$ and $\vf_n=\left(f_1,\cdots,f_n\right)^T$ for all $n\in\NN$.

\begin{proposition}\label{pro:pde-conv}
If the data $\left\{(\vL_n,\vf_n):n\in\NN\right\}$ are given by the set $\Lambda\times\Xi$ in equations \eqref{eq:pde-Lambda} and \eqref{eq:pde-Xi}, then
the kernel-based approximate function $\hat{u}_{\vL_n,\vf_n}$ converges to
the solution $u$ of elliptic equation \eqref{eq:pde} uniformly when $n\to\infty$.
\end{proposition}
\begin{proof}
Since $\Hilbert^m(\Domain)\subseteq\Cont(\Domain)$,
we know that $u,\hat{u}_{\vL_n,\vf_n}\in\Cont(\Domain)$. If we prove that
\[
\lim_{n\to\infty}\hat{u}_{\vL_n,\vf_n}(\vx)=u(\vx),\quad\text{for }\vx\in\Domain,
\]
then the compactness of $\Domain$ assures that
\[
\lim_{n\to\infty}\norm{\hat{u}_{\vL_n,\vf_n}-u}_{\infty}=0.
\]

Now we show the convergence of the kernel-based estimator $\hat{u}_{\vL_n,\vf_n}(\vx)$ in equation \eqref{eq:pde-kerappr}.
Let the interpolation events $\left\{\Aset_n:n\in\NN\right\}$ be constructed by the data $\left\{(\vL_n,\vf_n):n\in\NN\right\}$ same as Example \ref{exa:Hermite-Birkhoff}.
By the construction of $\Lambda\times\Xi$, the properties of mollification and the maximum principle
assure that $L_n\omega=f_n$ for all $n\in\NN$ if and only if $\omega=u$. This shows that
$\Aset_1\supseteq\ldots\supseteq\Aset_n\supseteq\ldots\supseteq
\Aset_{\infty}=\left\{u\right\}$. Therefore,
Theorem~\ref{thm:kerappr-conv}
guarantees that $s_{\vL_n,\vf_n}(\vx)$ converges to $u(\vx)$ when $n\to\infty$ by replacing $L=\delta_{\vx}$.
\end{proof}

%

In this article, we show how to solve the elliptic problems by the kernel-based probability measures.
We can construct and analyze the kernel-based approximate functions by the deterministic scattered data.
Different from the classical meshfree methods, we introduce the algorithms and theorems by the stochastic approaches. In the next article, we will apply the kernel-based probability measures to improve the meshfree methods for high-dimensional (stochastic) partial differential equations in \cite{CialencoFasshauerYe2012,DehghanShirzadi2015,GriebelRieger2017,HonSchabackZhong2013,KansaPowerFasshauerLing2004,Ling2006,LingSchaback2008}.

\section*{Acknowledgments}

The author would like to acknowledge support for this project from the ''Thousand Talents Program" of China, the Natural Science Foundation of
China (NSF grant 11601162), and the South China Normal University (school grants 671082,
S80835, S81031, and S81361).




\bibliographystyle{plain}
\bibliography{KernelYe}

\begin{thebibliography}{10}

\bibitem{BerlinetThomas2004}
A.~Berlinet and C.~Thomas-Agnan.
\newblock {\em Reproducing {K}ernel {H}ilbert {S}paces in {P}robability and
  {S}tatistics}.
\newblock Kluwer Academic Publishers, Boston, MA, 2004.

\bibitem{Buhmann2003}
M.~D. Buhmann.
\newblock {\em Radial {B}asis {F}unctions: {T}heory and {I}mplementations}.
\newblock Cambridge University Press, Cambridge, 2003.

\bibitem{CialencoFasshauerYe2012}
I.~Cialenco, G.~E. Fasshauer, and Q.~Ye.
\newblock Approximation of stochastic partial differential equations by a
  kernel-based collocation method.
\newblock {\em Int. J. Comput. Math.}, 89:2543--2561, 2012.

\bibitem{DehghanShirzadi2015}
M.~Dehghan and M.~Shirzadi.
\newblock Numerical solution of stochastic elliptic partial differential
  equations using the meshless method of radial basis functions.
\newblock {\em Eng. Anal. Bound. Elem.}, 50:291--303, 2015.

\bibitem{Fasshauer2007}
G.~E. Fasshauer.
\newblock {\em Meshfree {A}pproximation {M}ethods with {\sc{Matlab}}}.
\newblock World Scientific Publishing Co. Pte. Ltd., Hackensack, NJ, 2007.

\bibitem{FasshauerMcCourt2015}
G.~E. Fasshauer and M.~J. McCourt.
\newblock {\em Kernel-based Approximation Methods using {MATLAB}}.
\newblock World Scientific Publishing Co. Pte. Ltd., Hackensack, NJ, 2015.

\bibitem{FasshauerYe2012SPDEMCQMC}
G.~E. Fasshauer and Q.~Ye.
\newblock A kernel-based collocation method for elliptic partial differential
  equations with random coefficients.
\newblock In J.~Dick, F.~Y. Kuo, G.~W. Peters, and I.~H. Sloan, editors, {\em
  Monte Carlo and Quasi-Monte Carlo Methods 2012}, pages 331--348.
  Springer-Verlag, 2013.

\bibitem{FasshauerYe2012Bonn}
G.~E. Fasshauer and Q.~Ye.
\newblock Kernel-based collocation methods versus {G}alerkin finite element
  methods for approximating elliptic stochastic partial differential equations.
\newblock In M.~Griebel and M.~A. Schweitzer, editors, {\em Meshfree Methods
  for Partial Differential Equations VI}, pages 155--170. Springer-Verlag,
  2013.

\bibitem{GriebelRieger2017}
M.~Griebel and C.~Rieger.
\newblock Reproducing kernel {H}ilbert spaces for parametric partial
  differential equations.
\newblock {\em SIAM/ASA J. Uncertain. Quantif.}, 5(1):111--137, 2017.

\bibitem{HastieTibshiraniFriedman2009}
T.~Hastie, R.~Tibshirani, and J.~Friedman.
\newblock {\em The {E}lements of {S}tatistical {L}earning}.
\newblock Springer-Verlag, New York, second edition, 2009.

\bibitem{HonSchabackZhong2013}
Y.~C. Hon, R.~Schaback, and M.~Zhong.
\newblock The meshless {K}ernel-based method of lines for parabolic equations.
\newblock {\em Comput. Math. Appl.}, 68(12, part A):2057--2067, 2014.

\bibitem{KansaPowerFasshauerLing2004}
E.~J. Kansa, H.~Power, G.~E. Fasshauer, and L.~Ling.
\newblock A volumetric integral radial basis function method for time dependent
  partial differential equations.
\newblock {\em Eng. Anal. Bound. Elem.}, 28(10):1191--1206, 2004.

\bibitem{Kuelbs1976}
J.~Kuelbs.
\newblock A strong convergence theorem for {B}anach space valued random
  variables.
\newblock {\em Ann. Probability}, 4(5):744--771, 1976.

\bibitem{Ling2006}
L.~Ling.
\newblock Finding numerical derivatives for unstructured and noisy data by
  multiscale kernels.
\newblock {\em SIAM J. Numer. Anal.}, 44(4):1780--1800, 2006.

\bibitem{LingSchaback2008}
L.~Ling and R.~Schaback.
\newblock Stable and convergent unsymmetric meshless collocation methods.
\newblock {\em SIAM J. Numer. Anal.}, 46(3):1097--1115, 2008.

\bibitem{Matheron1973}
G.~Matheron.
\newblock The intrinsic random functions and their applications.
\newblock {\em Advances in Appl. Probability}, 5:439--468, 1973.

\bibitem{ScheuererSchabackSchlather2012}
M.~Scheuerer, R.~Schaback, and M.~Schlather.
\newblock Interpolation of spatial data - a stochastic or a deterministic
  problem?
\newblock {\em Eur. J. Appl. Math.}, 24:601--629, 2013.

\bibitem{ScholkopfSmola2002}
B.~Sch{\"o}lkopf and A.~J. Smola.
\newblock {\em Learning with Kernels: Support Vector Machines, Regularization,
  Optimization, and Beyond}.
\newblock MIT Press, Cambridge, MA, USA, December 2002.

\bibitem{SteinwartChristmann2008}
I.~Steinwart and A.~Christmann.
\newblock {\em Support {V}ector {M}achines}.
\newblock Springer-Verlag, New York, 2008.

\bibitem{Stroock1993}
D.~W. Stroock.
\newblock {\em Probability {T}heory, {A}n {A}nalytic {V}iew}.
\newblock Cambridge University Press, Cambridge, second edition, 2010.

\bibitem{Wendland2005}
H.~Wendland.
\newblock {\em Scattered {D}ata {A}pproximation}.
\newblock Cambridge University Press, Cambridge, 2005.

\bibitem{Wu1992}
Z.~Wu.
\newblock Hermite-birkhoff interpolation of scattered data by radial basis
  functions.
\newblock {\em Approx. Theory Appl.}, 8:1--10, 1992.

\bibitem{WuSchaback1993}
Z.~Wu and R.~Schaback.
\newblock Local error estimates for radial basis function interpolation of
  scattered data.
\newblock {\em IMA J. Numer. Anal.}, 13(1):13--27, 1993.

\bibitem{Ye2016PDK}
Q.~Ye.
\newblock Optimal designs of positive definite kernels for scattered data
  approximation.
\newblock {\em Appl. Comput. Harmon. Anal.}, 41(1):214--236, 2016.

\bibitem{Ye2017Kriging}
Q.~Ye.
\newblock Generalizations of kriging methods in spatial data analysis.
\newblock In M.~Griebel and M.~A. Schweitzer, editors, {\em Meshfree Methods
  for Partial Differential Equations VIII}, pages 145--166, Berlin, 2017.
  Springer.

\bibitem{Ye2017PDE}
Q.~Ye.
\newblock Kernel-based approximation methods for partial differential
  equations: deterministic or stochastic problems?
\newblock In G.~E. Fasshauer and L.~L. Schumaker, editors, {\em Approximation
  Theory XV: San Antonio 2016}, pages 375--398, New York, 2017. Springer.

\end{thebibliography}


%
%
%

\end{document}